%% file: main.tex
\documentclass[a4paper,11pt]{article}
\usepackage[margin=1.0truein]{geometry}
\input{packages}
\input{settings}
\input{macros}

\newcommand{\sigmainit}{\sigma_{\mathrm{init}}}

\title{Parameter-Free Cubic-Regularized Newton Method: Sharp Complexity and Generalized Smoothness}
\author[1]{Shaoying Fang}
\author[1]{Naoki Marumo\thanks{E-mail: \email{marumo.naoki@gmail.com}}}
\author[1,2]{Akiko Takeda}

\affil[1]{\normalsize Graduate School of Information Science and Technology, University of Tokyo, Tokyo, Japan}
\affil[2]{\normalsize RIKEN Center for Advanced Intelligence Project, Tokyo, Japan}

\date{\today}

\begin{document}

\maketitle

\begin{abstract}
  We analyze a variant of the cubic-regularized Newton method for nonconvex optimization.
  This variant is parameter-free in that it requires no prior knowledge of problem-dependent parameters.
  Under the generalized smoothness condition $\|\nabla^3 f(x)\| \leq L_0 + L_1 \|\nabla f(x)\|$, we derive an oracle complexity bound for finding an $(\epsilon, \delta)$-second-order stationary point.
  This assumption is weaker than the generalized smoothness conditions used in existing analyses of second-order methods, while the complexity bound improves upon existing guarantees for parameter-free second-order methods.
  In particular, when $L_1 = 0$, the bound matches the optimal dependence on $L_0$ as well as on $\epsilon$, $\delta$, and the initial function value gap, up to additive logarithmic terms.
  To establish this bound, we derive Taylor-type inequalities and prove their equivalence to the generalized smoothness condition.
\end{abstract}

\smallskip
\begin{description}[leftmargin=!,labelwidth=\widthof{\bfseries Keywords:}]
  \item[Keywords:]
  Nonconvex optimization,
  second-order method,
  cubic-regularized Newton method,
  parameter-free algorithm,
  oracle complexity,
  generalized smoothness
  \item[MSC2020:]
  90C26, 90C30, 90C60, 65K05
\end{description}
\smallskip

\section{Introduction}
This paper studies general nonconvex optimization problems of the form
\begin{align}
  \min_{x \in \R^d} f(x),
\end{align}
where $f \colon \R^d \to \R$ is three times continuously differentiable and bounded below.
We focus on second-order methods, which query an oracle at a given point $x \in \R^d$ to obtain $f(x)$, $\nabla f(x)$, and $\nabla^2 f(x)$.
For such methods, a standard goal is to find an $(\epsilon, \delta)$-second-order stationary point, that is, a point $x \in \R^d$ satisfying $\norm*{\nabla f(x)} \leq \epsilon$ and $\nabla^2 f(x) \succeq -\delta I$, where $\epsilon, \delta > 0$ are accuracy parameters.\footnote{
  Some existing results guarantee only first-order stationarity, that is, they aim to find a point $x \in \R^d$ satisfying $\norm*{\nabla f(x)} \leq \epsilon$.
  In our notation, this corresponds to the case $\delta = \infty$.
}

The seminal work by \citet{nesterov2006cubic} introduced the cubic-regularized Newton (CRN) method for this problem.
Under Lipschitz continuity of the Hessian, i.e.,
\begin{align}
  \norm*{\nabla^2 f(y) - \nabla^2 f(x)} \leq L_0 \norm*{y - x}
  \quad
  \text{for all $x, y \in \R^d$},
  \label{eq:lip_hess}
\end{align}
the method achieves an oracle complexity bound of
\begin{align}
  \O \prn*{\Delta \prn*{\frac{\sqrt{L_0}}{\epsilon^{3/2}} + \frac{L_0^2}{\delta^3}}},
  \label{eq:crn_complexity}
\end{align}
where $\Delta \coloneqq f(x_0) - \inf_{x \in \R^d} f(x)$.
This complexity bound is known to be optimal for second-order methods under this assumption, not only in its dependence on $\epsilon$ and $\delta$ but also in its dependence on $\Delta$ and $L_0$~\citep{carmon2020lower,arjevani2020second}.
This guarantee assumes that the Lipschitz constant $L_0$ in~\cref{eq:lip_hess} is known in advance and is used to set the cubic regularization parameter.
This result has motivated subsequent developments of second-order methods in several directions, including parameter-free algorithms and generalized smoothness.

\emph{Parameter-free} algorithms do not require prior knowledge of problem-dependent parameters such as $L_0$.
Many such second-order methods have been developed based on CRN, trust-region methods, and related frameworks~\citep{cartis2011adaptive,cartis2011adaptiveb,cartis2012complexity,cartis2020concise,grapiglia2017regularized,cartis2019universal,he2026newton,hamad2022consistently,semenov2026gradient,gratton2026fast,marumo2026general}.
Even when parameter-free methods attain the optimal dependence on $\epsilon$, many of them~\citep{cartis2011adaptive,cartis2011adaptiveb,cartis2012complexity,cartis2020concise,grapiglia2017regularized,cartis2019universal,he2026newton,gratton2026fast,marumo2026general} have worse dependence on $L_0$ than in the optimal bound~\cref{eq:crn_complexity}, as summarized in \cref{tab:parameter-free_complexity}.
In particular, when $L_0$ is small, the optimal bound suggests a substantially better complexity bound, but these parameter-free guarantees do not reflect this improvement.

\citet{hamad2022consistently} resolved this issue for first-order stationarity, i.e., $\delta = \infty$.
Their parameter-free method achieves the complexity bound
\begin{align}
  \O \prn*{\Delta \frac{\sqrt{L_0}}{\epsilon^{3/2}}} + \tilde \O(1),
  \label{eq:hamad_complexity}
\end{align}
where $\tilde \O(\cdot)$ hides logarithmic factors.
This result matches the optimal bound~\cref{eq:crn_complexity} when $\delta = \infty$, up to additive logarithmic terms.
Such logarithmic terms typically arise in complexity bounds when parameter-free algorithms adapt to unknown parameters.
We call a complexity bound \emph{sharp} if it matches the optimal dependence on both the accuracy and problem-dependent parameters, up to additive logarithmic terms.
The same work explicitly left open whether an analogous sharp bound can be achieved by a parameter-free method for finding an $(\epsilon, \delta)$-second-order stationary point.

The sharp bound \cref{eq:hamad_complexity} has also been extended beyond a Lipschitz-continuous Hessian.
Under the generalized smoothness assumption that
\begin{align}
  \norm*{\nabla^2 f(y) - \nabla^2 f(x)}
  \leq
  \prn*{L_0 + L_1 \norm*{\nabla f(x)}} \norm*{y - x}
  \label{eq:hessian_gen_smoothness_strong}
\end{align}
for all $x, y \in \R^d$, a parameter-free method achieves a bound of $\O \big(\Delta \big(\sqrt{L_0} \epsilon^{-3/2} + \sqrt{L_1} \epsilon^{-1}\big)\big) + \tilde \O(1)$ when $\delta = \infty$~\citep{semenov2026gradient}.
Although this assumption is weaker than \cref{eq:lip_hess}, it still excludes simple functions such as the quartic function in \cref{ex:quartic}.
A local version of \cref{eq:hessian_gen_smoothness_strong}, where the condition is imposed only on pairs $(x, y) \in \mathcal N_r \coloneqq \Set*{(x, y) \in \R^d \times \R^d}{\norm*{x - y} \leq r}$ for some $r > 0$, has also been studied~\citep{gratton2026fast}.
However, the resulting guarantee incurs an additional $\log \epsilon^{-1}$ factor and has worse dependence on $L_0$ and $L_1$, as shown in \cref{tab:parameter-free_complexity}.

\begin{table}[t]
  \centering
  \renewcommand{\arraystretch}{1.7}
  \caption{
    Parameter-free second-order methods for nonconvex optimization.
    All complexity bounds omit additive $\tilde \O(1)$ terms for brevity.
    The $\delta < \infty$ column indicates whether the method guarantees second-order stationarity.
  }
  \label{tab:parameter-free_complexity}
  \begin{threeparttable}
    \begin{tabular}{@{} l l c l @{}}
      \toprule
      Reference
      & Oracle complexity
      & $\delta < \infty$
      & Smoothness assumption \\
      \midrule
      \citep{cartis2011adaptive,cartis2011adaptiveb}
      & $\O \prn[\Big]{
        \Delta \frac{1 + L_0^{3/2}}{\epsilon^{3/2}}
      }$
      &
      & \cref{eq:lip_hess}\\
      \citep{cartis2012complexity,cartis2020concise}
      & $\O \prn[\Big]{
        \Delta \prn[\Big]{
          \frac{1 + L_0^{3/2}}{\epsilon^{3/2}}
          + \frac{1 + L_0^3}{\delta^3}
        }
      }$
      & $\checkmark$
      & \cref{eq:lip_hess}\\
      \citep{grapiglia2017regularized,cartis2019universal}
      & $\O \prn*{
        \Delta \frac{1 + \sqrt{L_0}}{\epsilon^{3/2}}
      } 
      $
      &
      & \cref{eq:lip_hess}\\
      \citep{he2026newton}
      & $\O \prn*{
        \Delta \prn*{
          \frac{1 + \sqrt{L_0}}{\epsilon^{3/2}}
          + \frac{1 + L_0^2}{\delta^3}
        }
      }$
      & $\checkmark$
      & \cref{eq:lip_hess}\\
      \citep{hamad2022consistently}
      & $\O \prn*{
        \Delta \frac{\sqrt{L_0}}{\epsilon^{3/2}}
        % + \log \frac{\norm*{\nabla f(x_0)}}{\epsilon}
      }$
      &
      & \cref{eq:lip_hess}\\
      \citep{semenov2026gradient}
      & $\O \prn*{
        \Delta \prn*{
          \frac{\sqrt{L_0}}{\epsilon^{3/2}}
          + \frac{\sqrt{L_1}}{\epsilon}
        }
      }$
      &
      & \cref{eq:hessian_gen_smoothness_strong} for all $x, y \in \R^d$\\
      \citep{gratton2026fast}\tnote{*}
      & $\tilde \O \prn*{
        \Delta \prn*{
          1
          + \frac{1}{r^2}
          + L_0^2
          + L_1^2
        } \prn*{
          \frac{1}{\epsilon^{3/2}}
          + \frac{1}{\delta^3}
        }
      }$
      & $\checkmark$
      & \cref{eq:hessian_gen_smoothness_strong} for all $(x, y) \in \mathcal N_r$\\
      \textbf{This work}
      & $\O \prn*{
        \Delta \prn*{
          \frac{\sqrt{L_0}}{\epsilon^{3/2}}
          + \frac{\sqrt{L_1}}{\epsilon}
          + \frac{L_0^2}{\delta^3}
          + \frac{L_1}{\delta}
        }
      }$
      & $\checkmark$
      & \cref{asm:our_hessian_gen_smoothness}\\
      \bottomrule
    \end{tabular}
    \begin{tablenotes}
      \small
      \item[*]
      $\mathcal N_r \coloneqq \Set*{(x, y) \in \R^d \times \R^d}{\norm*{x - y} \leq r}$.
      \citep{gratton2026fast} additionally assumes weak convexity on a sublevel set, i.e., $\nabla^2 f(x) \succeq - \gamma I$ for some $\gamma > 0$.
    \end{tablenotes}
  \end{threeparttable}
  \bigskip
\end{table}

\paragraph{Contributions.}
We analyze a parameter-free variant of CRN for finding an $(\epsilon, \delta)$-second-order stationary point under the following assumption.
\begin{assumption}
  \label{asm:main}
  Let $L_0 > 0$ and $L_1 \geq 0$ be constants.
  \begin{enuminasm}
    \item \label{asm:our_hessian_gen_smoothness}
    $\norm*{\nabla^3 f(x)} \leq L_0 + L_1 \norm*{\nabla f(x)}$ for all $x \in \R^d$.
    \item \label{asm:f_lower_bound}
    $\Delta \coloneqq f(x_0) - \inf_{x \in \R^d} f(x) < + \infty$.
  \end{enuminasm}
\end{assumption}
In contrast to the two-point condition~\cref{eq:hessian_gen_smoothness_strong}, \cref{asm:our_hessian_gen_smoothness} is a pointwise condition on the third derivative.
This assumption naturally generalizes Lipschitz continuity of the Hessian, since \cref{eq:lip_hess} is equivalent to $\norm*{\nabla^3 f(x)} \leq L_0$ for all $x \in \R^d$.
Under \cref{asm:main}, the CRN variant finds an $(\epsilon, \delta)$-second-order stationary point within at most
\begin{align}
  \O \prn*{
    \Delta \prn*{
      \frac{\sqrt{L_0}}{\epsilon^{3/2}}
      + \frac{\sqrt{L_1}}{\epsilon}
      + \frac{L_0^2}{\delta^3}
      + \frac{L_1}{\delta}
    }
  }
  + \tilde \O(1)
  \label{eq:main_complexity}
\end{align}
oracle calls.
Here, $\O(\cdot)$ hides only universal constants, and $\tilde \O(1)$ denotes additive logarithmic terms depending only on quantities other than $\epsilon$ and $\delta$.
The contributions are summarized as follows.
\begin{itemize}
  \item
  \textbf{Weaker assumption.}
  Our pointwise smoothness condition further relaxes the two-point generalized smoothness condition~\cref{eq:hessian_gen_smoothness_strong}, which is used in existing analyses of second-order methods~\citep{semenov2026gradient,gratton2026fast,xie2024trust}.
  Indeed, taking the limit $y \to x$ in~\cref{eq:hessian_gen_smoothness_strong} yields \cref{asm:our_hessian_gen_smoothness}.
  This relaxation is strict, as illustrated by \cref{ex:quartic,ex:smoothabs_xsinx}.
  \item
  \textbf{Sharp bound for second-order stationarity.}
  We close the open question raised by \citet{hamad2022consistently} for second-order stationarity.
  When $L_1 = 0$, our bound~\cref{eq:main_complexity} matches the optimal bound~\cref{eq:crn_complexity} in its dependence on $\Delta$, $L_0$, $\epsilon$, and $\delta$, up to additive logarithmic terms.
  To the best of our knowledge, this is the first sharp bound for parameter-free second-order methods, even in the classical case $L_1 = 0$.
  \item
  \textbf{Universality.}
  The guarantee is universal with respect to the choice of generalized smoothness constants.
  For a fixed function $f$, there may be many admissible pairs $(L_0, L_1)$ satisfying \cref{asm:our_hessian_gen_smoothness}.
  The complexity bound~\cref{eq:main_complexity} holds simultaneously for every such pair.
  Consequently, the method automatically enjoys the best bound obtained by optimizing the guarantee over all admissible pairs.
\end{itemize}
This universality is stronger than parameter-freeness alone.
A parameter-free algorithm may still maintain internal estimates of $L_0$ and $L_1$, in which case its guarantee may be tied to those estimates.
The CRN variant analyzed here does not estimate $(L_0, L_1)$; it only adaptively adjusts a single regularization parameter through backtracking.

\paragraph{Technical overview.}
Two ingredients are central to our analysis: Taylor-type characterizations of generalized smoothness and a refined analysis of backtracking.
Standard analyses of second-order methods rely on Taylor-type inequalities derived from \cref{eq:lip_hess}, which provide upper bounds on the function value and the gradient norm.
We first develop analogous inequalities under \cref{asm:our_hessian_gen_smoothness}.
Compared with the standard Taylor bounds, the resulting bounds are necessarily weaker: they involve hyperbolic-function factors and additional dependence on Hessian information.
A key difficulty is that the subsequent analysis must work with these weaker inequalities.
Since these inequalities are equivalent to \cref{asm:our_hessian_gen_smoothness}, the hyperbolic factors and the additional Hessian dependence are not artifacts of the proof; rather, they are the price of working under this weaker assumption.

The second ingredient is a refined analysis of backtracking.
In CRN methods, the regularization parameter $\sigma > 0$ is commonly adjusted by checking a sufficient decrease condition.
We use a similar backtracking procedure, but we check two acceptance conditions: one for function decrease and one for the gradient norm at the trial point.
In addition, once a value of $\sigma$ has been accepted, the next backtracking search is initialized at $\sigma / 2$.
Although such a halving step is often used to improve practical performance, it plays an essential role in our complexity analysis.
Together, these two ingredients enable us to obtain the bound~\cref{eq:main_complexity} under our weaker generalized smoothness assumption.

\paragraph{Notation.}
For vectors $a, b \in \R^d$, let $\inner*{a}{b}$ denote the Euclidean inner product and let $\norm*{a}$ denote the Euclidean norm.
For a matrix $A \in \R^{m \times n}$, let $\norm*{A}$ denote the operator norm, i.e., the largest singular value of $A$.
For a symmetric third-order tensor $T$, define
\begin{align}
  \norm*{T} \coloneqq \sup_{\norm*{u_1} = \norm*{u_2} = \norm*{u_3} = 1} \abs*{T\brk*{u_1, u_2, u_3}} = \sup_{\norm*{u} = 1} \abs*{T\brk*{u, u, u}}.
  \label{eq:def_tensor_norm}
\end{align}
The second equality follows from \citep[Eq.~(A.7)]{cartis2022evaluation}.

\section{Examples of generalized smoothness}
\label{sec:examples}
We present concrete examples of functions to illustrate the scope of \cref{asm:our_hessian_gen_smoothness}.
The first example satisfies \cref{asm:our_hessian_gen_smoothness} but not the global condition in~\cref{eq:hessian_gen_smoothness_strong}.
The second example gives a stronger separation: it satisfies \cref{asm:our_hessian_gen_smoothness}, whereas \cref{eq:hessian_gen_smoothness_strong} fails even locally.

\begin{example}
  \label{ex:quartic}
  Consider the one-dimensional function $f(x) = \frac{1}{4} x^4$.
  Since $f'(x) = x^3$, $f''(x) = 3 x^2$, and $f'''(x) = 6 x$, we have, for any $\alpha > 0$,
  \begin{align}
    \abs*{f'''(x)}
    =
    6 \abs*{f'(x)}^{1/3}
    \leq
    2 \prn*{
      \frac{\abs*{f'(x)}}{\alpha^2}
      + 2 \alpha
    },
  \end{align}
  where we used the AM--GM inequality $t = (\frac{t^3}{\alpha^2} \cdot \alpha \cdot \alpha)^{1/3} \leq \frac{1}{3} \prn[\big]{\frac{t^3}{\alpha^2} + \alpha + \alpha}$ for $t = \abs*{f'(x)}^{1/3}$.
  Hence, $f$ satisfies \cref{asm:our_hessian_gen_smoothness} with $(L_0, L_1) = (4 \alpha, 2 / \alpha^2)$ for any $\alpha > 0$.

  In contrast, there does not exist $(L_0, L_1)$ such that \cref{eq:hessian_gen_smoothness_strong} holds for all $x, y \in \R$.
  Indeed, if $x = 0$, then \cref{eq:hessian_gen_smoothness_strong} reduces to $\abs*{3 y^2} \leq L_0 \abs*{y}$, which is violated for sufficiently large $\abs*{y}$.
\end{example}
\begin{example}
  \label{ex:smoothabs_xsinx}
  Consider the one-dimensional function $f(x) = \sqrt{1 + x^2} - x \cos x$.
  Since
  \begin{align}
    f'(x) &= x \prn*{1 + x^2}^{-1/2} - \cos x + x \sin x,\\
    f''(x) &= \prn*{1 + x^2}^{-3/2} + 2 \sin x + x \cos x,\\
    f'''(x) &= -3 x \prn*{1 + x^2}^{-5/2} + 3 \cos x - x \sin x,
  \end{align}
  we have
  \begin{align}
    f'''(x)
    =
    - f'(x)
    + x \prn*{1 + x^2}^{-1/2} - 3 x \prn*{1 + x^2}^{-5/2}
    + 2 \cos x.
  \end{align}
  Hence, we have
  \begin{align}
    \abs*{f'''(x)}
    \leq
    \abs*{f'(x)}
    + \abs*{x} \prn*{1 + x^2}^{-1/2}
    + 3 \abs*{x} \prn*{1 + x^2}^{-5/2}
    + 2 \abs*{\cos x}
    \leq
    \abs*{f'(x)} + 6,
  \end{align}
  which shows that $f$ satisfies \cref{asm:our_hessian_gen_smoothness} with $(L_0, L_1) = (6, 1)$.

  In contrast, there does not exist $(L_0, L_1)$ and $r > 0$ such that \cref{eq:hessian_gen_smoothness_strong} holds for all $(x, y) \in \mathcal N_r$.
  For $n \in \N$, let $x_n = 2 n \pi$ and $y_n = x_n + h$, where $h \in (0, r]$ is a small constant.
  Then we have
  \begin{align}
    f''(y_n) - f''(x_n)
    &=
    \brk*{\prn*{1 + y_n^2}^{-3/2} + 2 \sin h + y_n \cos h}
    - \brk*{\prn*{1 + x_n^2}^{-3/2} + x_n}\\
    &=
    (\cos h - 1) x_n + \O(1).
  \end{align}
  Taking $n \to \infty$ yields $\abs*{f''(y_n) - f''(x_n)} \to \infty$, which shows that \cref{eq:hessian_gen_smoothness_strong} cannot hold.
\end{example}

We next consider two benchmark functions for which explicit constants $(L_0, L_1)$ are less transparent.
Although \cref{asm:our_hessian_gen_smoothness} is imposed for every $x \in \R^d$, fixing a point $x \in \R^d$ and computing $(u, v) = \prn{\norm{\nabla f(x)}, \norm{\nabla^3 f(x)}}$ gives a necessary condition on a pair $(L_0, L_1)$, namely $v \leq L_0 + L_1 u$.
By randomly sampling points $x \in \R^d$ and plotting the corresponding pairs $(u, v)$, we can visually inspect which values of $(L_0, L_1)$ are plausible for the function.

\begin{example}
  Consider the Rastrigin function \citep{rastrigin1974extremal}:
  \begin{align}
    f(x_1, x_2)
    &=
    20 + x_1^2 + x_2^2 - 10 \cos(2 \pi x_1) - 10 \cos(2 \pi x_2).
  \end{align}
  The heat map on the left of \cref{fig:rastrigin} shows the function values.
  The white crosses indicate the randomly sampled points $x \in \R^2$.
  The scatter plot on the right shows the pairs $(u, v) = \prn{\norm{\nabla f(x)}, \norm{\nabla^3 f(x)}}$ at the sampled points.
  The plot shows that all sampled pairs $(u, v)$ lie below the affine line $v = 100 + 50 u$.
  This indicates that $(L_0, L_1) \simeq (100, 50)$ is a reasonable scale for \cref{asm:our_hessian_gen_smoothness}.

  In contrast, the classical assumption \cref{eq:lip_hess} requires all points $(u, v)$ in the scatter plot to lie below a horizontal line $v = L_0$.
  This would require roughly $L_0 \simeq 2500$.
  Thus, \cref{asm:our_hessian_gen_smoothness} allows a much smaller value of $L_0$ in this comparison, which suggests improvements in the factors multiplying the $\epsilon^{-3/2}$ and $\delta^{-3}$ terms in the bound~\cref{eq:main_complexity}.
\end{example}

\begin{figure}[t]
  \centering
  \begin{subfigure}[t]{\textwidth}
    \centering
    \includegraphics[width=0.48\textwidth]{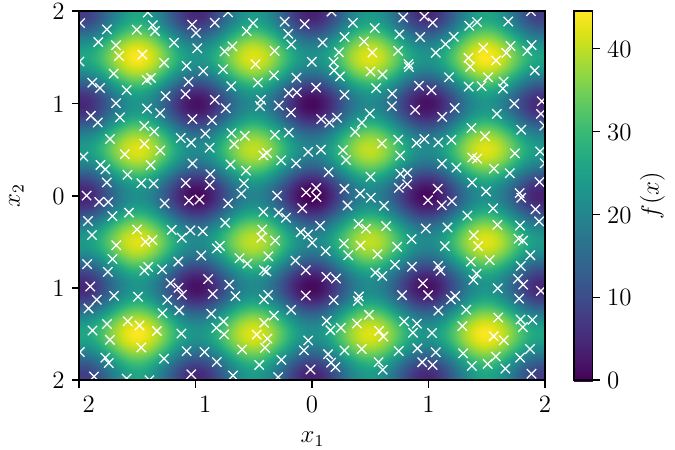}
    \hfill
    \includegraphics[width=0.48\textwidth]{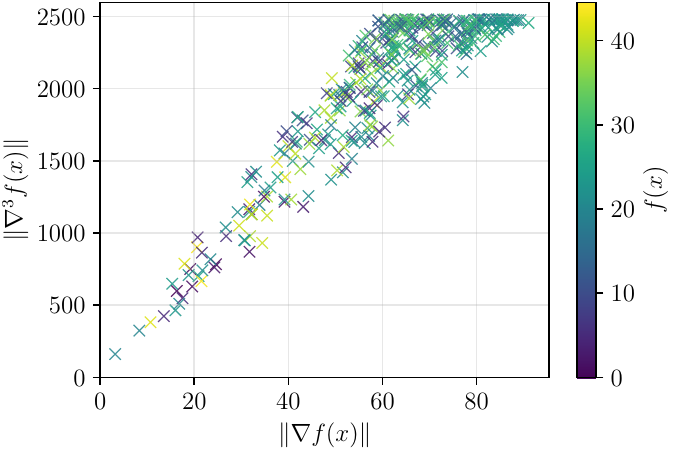}
    \caption{Rastrigin function\label{fig:rastrigin}}
  \end{subfigure}
  \par\bigskip
  \begin{subfigure}[t]{\textwidth}
    \centering
    \includegraphics[width=0.48\textwidth]{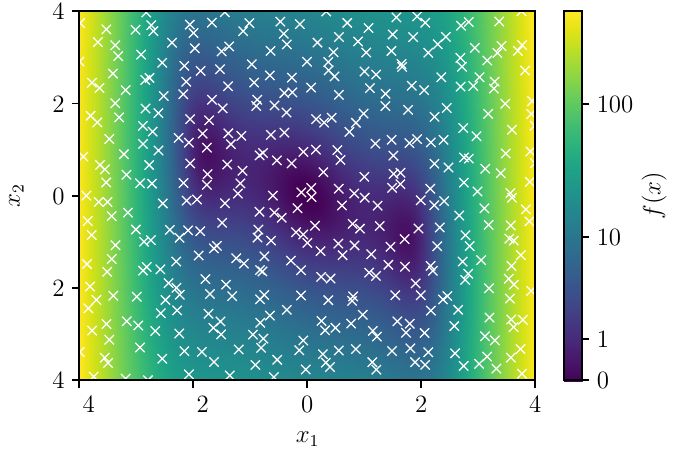}
    \hfill
    \includegraphics[width=0.48\textwidth]{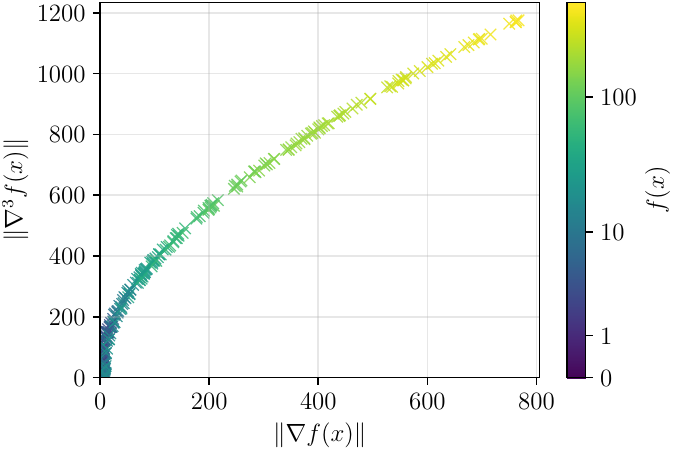}
    \caption{Three-hump camel function\label{fig:camel}}
  \end{subfigure}
  \caption{Benchmark functions and corresponding scatter plots of $\prn{\norm{\nabla f(x)}, \norm{\nabla^3 f(x)}}$ at randomly sampled points $x \in \R^2$.}
\end{figure}

\begin{example}
  Consider the three-hump camel function \citep{branin1972widely}:
  \begin{align}
    f(x_1, x_2)
    &=
    2 x_1^2 - 1.05 x_1^4 + \frac{1}{6} x_1^6 + x_1 x_2 + x_2^2.
  \end{align}
  As in the previous example, \cref{fig:camel} suggests that \cref{asm:our_hessian_gen_smoothness} can allow a much smaller value of $L_0$ than the classical assumption.
  A further point illustrated by this example is the trade-off between $L_0$ and $L_1$.
  For example, the scatter plot suggests that both $(L_0, L_1) \simeq (500, 1)$ and $(L_0, L_1) \simeq (200, 3)$ are consistent with \cref{asm:our_hessian_gen_smoothness}.
  Which of these two pairs gives a better bound in \cref{eq:main_complexity} depends on the values of $\epsilon$ and $\delta$.
  Since our guarantee is universal, it automatically achieves the best bound obtained from such pairs $(L_0, L_1)$.
\end{example}

\section{Related work}
This section reviews related work from several perspectives and positions our contribution within the literature.

\paragraph{Second-order methods for nonconvex optimization.}
Many second-order methods are based on Newton's method.
Since plain Newton steps are not globally reliable in nonconvex optimization, globalization techniques such as regularization and trust-region methods have been developed.
The cubic-regularized Newton (CRN) method \citep{nesterov2006cubic} achieves the oracle complexity~\cref{eq:crn_complexity} under Lipschitz continuity of the Hessian.
Following this work, many second-order methods have been developed; some methods guarantee $\epsilon$-first-order stationarity~\citep{birgin2017worstcase,grapiglia2017regularized,cartis2019universal,curtis2017trust}, while others guarantee $(\epsilon, \delta)$-second-order stationarity~\citep{cartis2011adaptiveb,cartis2012complexity,cartis2020concise,curtis2019inexact,royer2018complexity,royer2020newton,curtis2021trust}.
The term $\O \prn*{\Delta \sqrt{L_0} \epsilon^{-3/2}}$ in~\cref{eq:crn_complexity} cannot be improved~\citep{carmon2020lower}, and the same is true for the term $\O \prn*{\Delta L_0^2 \delta^{-3}}$~\citep{arjevani2020second}.
Thus, the bound~\cref{eq:crn_complexity} is optimal in its dependence on $\Delta$, $L_0$, $\epsilon$, and $\delta$.

We show that a parameter-free variant of CRN achieves the oracle complexity~\cref{eq:main_complexity} under the weaker smoothness condition in \cref{asm:our_hessian_gen_smoothness}.
In particular, when $L_1 = 0$, the resulting guarantee matches the optimal bound~\cref{eq:crn_complexity}, up to additive logarithmic terms.

\paragraph{Generalized smoothness.}
The concept of generalized smoothness was introduced by \citet{zhang2020why} in the analysis of first-order methods.
They analyzed a stochastic gradient method under the following smoothness assumption:
\begin{align}
  \norm*{\nabla^2 f(x)} \leq M_0 + M_1 \norm*{\nabla f(x)}
  \quad
  \text{for all $x \in \R^d$.}
  \label{eq:first_order_gen_smoothness}
\end{align}
Motivated by the practical relevance of this assumption, many first-order methods have since been studied under generalized smoothness, including deterministic methods~\citep{zhang2020improved,chen2023generalized,koloskova2023revisiting,vankov2025optimizing,takezawa2024parameter,gorbunov2025methods,oikonomidis2025nonlinearly} and stochastic methods~\citep{zhang2020improved,chen2023generalized,koloskova2023revisiting,reisizadeh2025variance,faw2023beyond}.
Several variants of \cref{eq:first_order_gen_smoothness} have also been proposed~\citep{chen2023generalized,li2023convex,hubler2024parameter}.
\Cref{asm:our_hessian_gen_smoothness} can be viewed as a natural second-order analogue of \cref{eq:first_order_gen_smoothness}.

Second-order generalized smoothness was introduced in \citep{xie2024trust}.
That work assumes the two-point condition~\cref{eq:hessian_gen_smoothness_strong} for all $(x, y) \in \mathcal N_r$ and analyzes a trust-region method for stochastic nonconvex optimization.
A Newton-type method was later analyzed under the same local condition, together with weak convexity on a sublevel set~\citep{gratton2026fast}.
A different framework, called gradient-normalized smoothness, was studied in~\citep{semenov2026gradient}; the analysis yields a guarantee under the stronger global version of \cref{eq:hessian_gen_smoothness_strong} imposed for all $x, y \in \R^d$.
\Cref{asm:our_hessian_gen_smoothness} is weaker than both the local and global two-point conditions above.

\paragraph{Parameter-free methods.}
Problem-dependent parameters such as Lipschitz constants are often unknown in practice, which makes algorithms that do not require their values attractive.
This has motivated extensive study of parameter-free algorithms; see, e.g., \citep{marumo2024parameter,marumo2025universal,marumo2025parameter,xiong2026parameter,kong2024complexity,yagishita2025simple} for first-order methods and \citep{cartis2011adaptive,cartis2011adaptiveb,cartis2012complexity,cartis2020concise,grapiglia2017regularized,he2026newton,cartis2019universal,hamad2022consistently,semenov2026gradient,gratton2026fast,marumo2026general} for second-order methods.
Existing guarantees for parameter-free second-order methods often recover the optimal dependence on the accuracy parameters but lose the optimal dependence on problem-dependent constants.
Notable exceptions are \citep{hamad2022consistently,semenov2026gradient}, which achieve the optimal dependence on problem-dependent constants up to additive logarithmic terms.
However, both results guarantee only $\epsilon$-first-order stationarity, as shown in \cref{tab:parameter-free_complexity}.
Our result gives a parameter-free method for the stronger goal of finding an $(\epsilon, \delta)$-second-order stationary point.

\section{Taylor-type characterizations of generalized smoothness}
This section presents Taylor-type inequalities that are equivalent to second-order generalized smoothness.

We use the standard hyperbolic functions $\sinh t \coloneqq \prn*{\e^t - \e^{-t}} / 2$ and $\cosh t \coloneqq \prn*{\e^t + \e^{-t}} / 2$.
Throughout the paper, quotients involving these hyperbolic functions that have removable singularities at $t = 0$ are understood by continuous extension.
In particular,
\begin{align}
  \lim_{t \to 0} \frac{\sinh t}{t}
  &=
  1,\qquad
  \lim_{t \to 0} \frac{\cosh t - 1}{t^2}
  =
  \frac{1}{2},\qquad
  \lim_{t \to 0} \frac{\sinh t - t}{t^3}
  =
  \frac{1}{6}.
  \label{eq:hyperbolic_limits}
\end{align}

For $x, s \in \R^d$, define
\begin{align}
  \tilde L(x, s)
  \coloneqq{}&
  L_0 + L_1 \max_{t \in [0, 1]} \norm*{\nabla f(x) + t \nabla^2 f(x) s}\\
  ={}&
  L_0 + L_1 \max \set*{\norm*{\nabla f(x)},\, \norm*{\nabla f(x) + \nabla^2 f(x) s}}.
\end{align}
This quantity can be viewed as a local Lipschitz constant for the Hessian along the line segment from $x$ to $x+s$; it plays an important role in our analysis.

The following theorem gives the Taylor-type inequalities that are equivalent to \cref{asm:our_hessian_gen_smoothness}.
The proof is inspired by \citep[Lemma~2.2]{vankov2025optimizing}, which establishes analogous inequalities under the first-order generalized smoothness condition~\cref{eq:first_order_gen_smoothness}.

\begin{theorem}
  \label{thm:taylor_equivalence}
  Let $L_0, L_1 \geq 0$ and let $f \colon \R^d \to \R$ be a three times continuously differentiable function.
  The following statements are equivalent.
  In the statements involving $s \in \R^d$, write $\alpha \coloneqq \sqrt{L_1} \norm*{s}$.
  \begin{itemize}
    \item
    For all $x \in \R^d$,
    \begin{align}
      \norm*{\nabla^3 f(x)} \leq L_0 + L_1 \norm*{\nabla f(x)}.
      \label{eq:hessian_gen_bounded}
    \end{align}
    \item For all $x, s \in \R^d$,
    \begin{align}
      \norm*{\nabla^2 f(x + s) - \nabla^2 f(x)}
      &\leq
      \tilde L(x, s) \norm*{s} \frac{\sinh \alpha}{\alpha}.
      \label{eq:hessian_taylor_G_bound}
    \end{align}

    \item For all $x, s \in \R^d$,
    \begin{align}
      \norm*{\nabla f(x + s) - \nabla f(x) - \nabla^2 f(x) s}
      &\leq
      \tilde L(x, s) \norm*{s}^2 \frac{\cosh \alpha - 1}{\alpha^2}.
      \label{eq:grad_taylor_G_bound}
    \end{align}

    \item For all $x, s \in \R^d$,
    \begin{align}
      \abs*{f(x + s) - f(x) - \inner*{\nabla f(x)}{s} - \frac{1}{2} \inner*{\nabla^2 f(x) s}{s}}
      &\leq
      \tilde L(x, s) \norm*{s}^3 \frac{\sinh \alpha - \alpha}{\alpha^3}.
      \label{eq:function_taylor_bound}
    \end{align}
  \end{itemize}
\end{theorem}

\begin{proof}
  We prove the following chain of implications: $\text{\cref{eq:hessian_gen_bounded}} \implies \text{\cref{eq:hessian_taylor_G_bound}} \implies \text{\cref{eq:grad_taylor_G_bound}} \implies \text{\cref{eq:function_taylor_bound}} \implies \text{\cref{eq:hessian_gen_bounded}}$.

  \paragraph{\cref{eq:hessian_gen_bounded} $\implies$ \cref{eq:hessian_taylor_G_bound}.}
  Taylor's theorem gives
  \begin{align}
    \nabla f(x + t s)
    &=
    \nabla f(x)
    + t \nabla^2 f(x) s
    + \int_0^t \int_0^\tau \nabla^3 f(x + \rho s)[s, s] \,\dd \rho \dd \tau,
    \label{eq:equiv_proof_taylor_1}
  \end{align}
  and hence, for all $t \in [0, 1]$,
  \begin{align}
    \norm*{\nabla f(x + t s)}
    &\leq
    \max_{\tau \in [0, 1]} \norm*{\nabla f(x) + \tau \nabla^2 f(x) s}
    + \norm*{s}^2 \int_0^t \int_0^\tau \norm*{\nabla^3 f(x + \rho s)} \,\dd \rho \dd \tau.
  \end{align}
  Plugging this inequality into $\norm*{\nabla^3 f(x + ts)} \leq L_0 + L_1 \norm*{\nabla f(x + ts)}$ and using the definitions of $\tilde L(x, s)$ and $\alpha$ yields
  \begin{align}
    \norm*{\nabla^3 f(x + t s)}
    &\leq
    \tilde L(x, s) + \alpha^2 \int_0^t \int_0^\tau \norm*{\nabla^3 f(x + \rho s)} \,\dd \rho \dd \tau
    \eqqcolon \phi(t).
    \label{eq:def_phi_and_ode}
  \end{align}
  Since $\phi'(t) = \alpha^2 \int_0^t \norm*{\nabla^3 f(x + \rho s)} \,\dd \rho$ and $\phi''(t) = \alpha^2 \norm*{\nabla^3 f(x + t s)}$, the above inequality implies $\phi''(t) \leq \alpha^2 \phi(t)$.
  We next derive an upper bound on $\phi(t)$ from this differential inequality.

  Observe that
  \begin{align}
    \frac{\dd}{\dd t} \brk*{
      \e^{- \alpha t} \prn*{ \phi'(t) + \alpha \phi(t) }
    }
    =
    \e^{- \alpha t} \prn*{ \phi''(t) - \alpha^2 \phi(t) }
    \leq
    0.
  \end{align}
  Integrating this inequality gives $\int_0^t \frac{\dd}{\dd \tau} \brk*{ \e^{-\alpha \tau} \prn*{\phi'(\tau) + \alpha \phi(\tau)}} \dd \tau \leq 0$, and hence
  \begin{align}
    \e^{- \alpha t} \prn*{ \phi'(t) + \alpha \phi(t) }
    \leq
    \phi'(0) + \alpha \phi(0)
    =
    \alpha \tilde L(x, s),
  \end{align}
  where we have used $\phi(0) = \tilde L(x, s)$ and $\phi'(0) = 0$.
  Multiplying both sides by $\e^{2 \alpha t}$ gives
  \begin{align}
    \frac{\dd}{\dd t} \prn*{ \e^{\alpha t} \phi(t) }
    =
    \e^{\alpha t} \prn*{ \phi'(t) + \alpha \phi(t) }
    \leq
    \alpha \tilde L(x, s) \e^{2 \alpha t}.
  \end{align}
  Again, integrating this inequality yields
  \begin{align}
    \e^{\alpha t} \phi(t)
    \leq
    \phi(0) + \alpha \tilde L(x, s) \int_0^t \e^{2 \alpha \tau} \dd \tau
    =
    \tilde L(x, s) \frac{\e^{2 \alpha t} + 1}{2},
  \end{align}
  where we have used $\phi(0) = \tilde L(x, s)$.
  We thus obtain $\phi(t) \leq \tilde L(x, s) \cosh(\alpha t)$.
  Hence, \cref{eq:def_phi_and_ode} gives
  \begin{align}
    \norm*{\nabla^3 f(x + t s)}
    \leq
    \tilde L(x, s) \cosh(\alpha t).
  \end{align}

  The fundamental theorem of calculus gives
  \begin{align}
    \nabla^2 f(x + s) - \nabla^2 f(x)
    &=
    \int_0^1 \nabla^3 f(x + t s)[s] \,\dd t,
  \end{align}
  and hence
  \begin{align}
    \norm*{\nabla^2 f(x + s) - \nabla^2 f(x)}
    &\leq
    \norm*{s} \int_0^1 \norm*{\nabla^3 f(x + t s)} \,\dd t\\
    &\leq
    \tilde L(x, s) \norm*{s} \int_0^1 \cosh(\alpha t) \,\dd t
    =
    \tilde L(x, s) \norm*{s} \frac{\sinh \alpha}{\alpha},
  \end{align}
  which proves \cref{eq:hessian_taylor_G_bound}.

  \paragraph{\cref{eq:hessian_taylor_G_bound} $\implies$ \cref{eq:grad_taylor_G_bound}.}
  The fundamental theorem of calculus gives
  \begin{align}
    \nabla f(x + s) - \nabla f(x) - \nabla^2 f(x) s
    =
    \int_0^1 \prn*{\nabla^2 f(x + t s) - \nabla^2 f(x)} s,
  \end{align}
  and hence
  \begin{align}
    \norm*{\nabla f(x + s) - \nabla f(x) - \nabla^2 f(x) s}
    &\leq
    \norm*{s} \int_0^1 \norm*{\nabla^2 f(x + t s) - \nabla^2 f(x)} \,\dd t.
  \end{align}
  For $t \in [0, 1]$, applying \cref{eq:hessian_taylor_G_bound} with $s$ replaced by $ts$ and hence $\alpha$ replaced by $\alpha t$, we obtain
  \begin{align}
    \norm*{\nabla^2 f(x + t s) - \nabla^2 f(x)}
    \leq
    \tilde L(x, ts) \norm*{s} \frac{\sinh(\alpha t)}{\alpha}
    \leq
    \tilde L(x, s) \norm*{s} \frac{\sinh(\alpha t)}{\alpha}.
  \end{align}
  Combining these inequalities and using $\int_0^1 \sinh(\alpha t) \,\dd t = \frac{\cosh \alpha - 1}{\alpha}$ proves \cref{eq:grad_taylor_G_bound}.

  \paragraph{\cref{eq:grad_taylor_G_bound} $\implies$ \cref{eq:function_taylor_bound}.}
  The fundamental theorem of calculus gives
  \begin{align}
    f(x + s) - f(x) - \inner*{\nabla f(x)}{s} - \frac{1}{2} \inner*{\nabla^2 f(x) s}{s}
    &=
    \int_0^1 \inner*{\nabla f(x + t s) - \nabla f(x) - t \nabla^2 f(x) s}{s} \,\dd t,
  \end{align}
  and hence
  \begin{align}
    &\mathInd
    \abs*{f(x + s) - f(x) - \inner*{\nabla f(x)}{s} - \frac{1}{2} \inner*{\nabla^2 f(x) s}{s}}\\
    &\leq
    \norm*{s} \int_0^1 \norm*{\nabla f(x + t s) - \nabla f(x) - t \nabla^2 f(x) s} \,\dd t.
  \end{align}
  For $t \in [0, 1]$, replacing $s$ with $ts$ in~\cref{eq:grad_taylor_G_bound} gives
  \begin{align}
    \norm*{\nabla f(x + t s) - \nabla f(x) - t \nabla^2 f(x) s}
    \leq
    \tilde L(x, ts) \norm*{s}^2 \frac{\cosh(\alpha t) - 1}{\alpha^2}
    \leq
    \tilde L(x, s) \norm*{s}^2 \frac{\cosh(\alpha t) - 1}{\alpha^2}.
  \end{align}
  Combining these inequalities and using $\int_0^1 \prn*{\cosh(\alpha t) - 1} \,\dd t = \frac{\sinh \alpha - \alpha}{\alpha}$ proves \cref{eq:function_taylor_bound}.

  \paragraph{\cref{eq:function_taylor_bound} $\implies$ \cref{eq:hessian_gen_bounded}.}
  Let $u \in \R^d$ satisfy $\norm*{u} = 1$.
  For $t > 0$, replacing $s$ with $t u$ in~\cref{eq:function_taylor_bound} and dividing by $t^3$, we obtain
  \begin{align}
    \frac{1}{t^3} \abs*{
      f(x + t u) - f(x) - \inner*{\nabla f(x)}{t u} - \frac{1}{2} \inner*{\nabla^2 f(x) t u}{t u}
    }
    \leq
    \tilde L(x, t u)
    \frac{\sinh(\sqrt{L_1} t) - (\sqrt{L_1} t)}{(\sqrt{L_1} t)^3}.
  \end{align}
  Taking the limit as $t \to 0$ and using the third equation in~\cref{eq:hyperbolic_limits} gives
  \begin{align}
    \frac{1}{6} \abs*{\nabla^3 f(x)\brk*{u, u, u}}
    \leq
    \tilde L(x, \0) \cdot \frac{1}{6}
    =
    \frac{1}{6} \prn*{L_0 + L_1 \norm*{\nabla f(x)}}.
    \label{eq:third_derivative_bound}
  \end{align}
  Taking the supremum over $u$ and using \cref{eq:def_tensor_norm} proves \cref{eq:hessian_gen_bounded}.
\end{proof}

When $L_1 = 0$, we have $\tilde L(x, s) = L_0$ and $\alpha = 0$.
Thus, by \cref{eq:hyperbolic_limits}, the Taylor-type bounds~\cref{eq:hessian_taylor_G_bound,eq:grad_taylor_G_bound,eq:function_taylor_bound} reduce to the standard bounds under \cref{eq:lip_hess}:
\begin{align}
  \norm*{\nabla^2 f(x + s) - \nabla^2 f(x)}
  &\leq
  L_0 \norm*{s},\\
  \norm*{\nabla f(x + s) - \nabla f(x) - \nabla^2 f(x) s}
  &\leq
  \frac{L_0}{2} \norm*{s}^2,\\
  \abs*{f(x + s) - f(x) - \inner*{\nabla f(x)}{s} - \frac{1}{2} \inner*{\nabla^2 f(x) s}{s}}
  &\leq
  \frac{L_0}{6} \norm*{s}^3.
\end{align}

We also compare the Hessian-difference bound~\cref{eq:hessian_taylor_G_bound} with the existing two-point generalized smoothness condition~\cref{eq:hessian_gen_smoothness_strong}.
Setting $y = x + s$ in \cref{eq:hessian_taylor_G_bound} gives
\begin{align}
  \norm*{\nabla^2 f(y) - \nabla^2 f(x)}
  &\leq
  \prn[\Big]{
    L_0 + L_1 \max_{t \in [0, 1]}
    \norm*{\nabla f(x) + t \nabla^2 f(x) (y - x)}
  }
  \frac{\sinh \prn*{\sqrt{L_1} \norm*{y - x}}}{\sqrt{L_1}},
\end{align}
which is weaker than \cref{eq:hessian_gen_smoothness_strong}.
The equivalence in \cref{thm:taylor_equivalence} shows that this weaker bound is intrinsic to \cref{asm:our_hessian_gen_smoothness}, and this is one of the main difficulties in our analysis.

\section{Cubic-regularized Newton method with backtracking}
This section describes the method used in our analysis, summarized in \cref{alg:cubic_newton_backtracking}.
The algorithm itself is not new; rather, it is a variant of the standard cubic-regularized Newton method~\citep{nesterov2006cubic,cartis2011adaptive,cartis2011adaptiveb} with a minor modification to the backtracking conditions.
Moreover, the method is parameter-free in the sense that it requires no prior knowledge of problem-dependent parameters such as $L_0$ and $L_1$, while still achieving the desired complexity bound.

The algorithm generates a sequence of iterates $(x_k)_{k \in \N}$.
In the remainder of the paper, we use the notation
\begin{align}
  g_k \coloneqq \nabla f(x_k), \qquad H_k \coloneqq \nabla^2 f(x_k), \qquad
  \mu_k \coloneqq \max \set*{-\lambda_{\min}(H_k),\, 0},
  \label{eq:def_g_H_mu}
\end{align}
where $\lambda_{\min}(\cdot)$ denotes the minimum eigenvalue.

\begin{algorithm}[t]
  \caption{Cubic-regularized Newton method with backtracking}
  \label{alg:cubic_newton_backtracking}
  \begin{algorithmic}[1]
    \Require $x_0 \in \R^d$, $\sigmainit > 0$
    \State $\sigma \gets \sigmainit$
    \For{$k = 0, 1, 2, \ldots$}
      \State $g_k \gets \nabla f(x_k)$, $H_k \gets \nabla^2 f(x_k)$
      \label{alg-line:compute_g_H}
      \Loop
        \State Compute a trial step $s \in \R^d$ satisfying \cref{eq:subproblem_fo_condition,eq:subproblem_so_condition} by approximately solving \cref{eq:subproblem}
        \If{\cref{eq:accept_decrease_condition,eq:accept_grad_condition} hold}
        \label{alg-line:check_conditions}
          \State $(\sigma_k, s_k) \gets (\sigma, s)$
          \Comment{Accept $(\sigma, s)$}
          \State \textbf{break}
        \Else
          \State $\sigma \gets 2 \sigma$
        \EndIf
      \EndLoop
      \State $x_{k+1} \gets x_k + s_k$
      \State $\sigma \gets \sigma / 2$
      \label{alg-line:sigma_halving}
    \EndFor
  \end{algorithmic}
\end{algorithm}

\subsection{Algorithm}
At iteration $k$, the cubic-regularized Newton method computes the step $s_k \in \R^d$ by approximately solving the following subproblem:
\begin{align}
  \min_{s \in \R^d} \ 
  \set*{\inner*{g_k}{s} + \frac{1}{2} \inner*{H_k s}{s} + \frac{\sigma}{3} \norm*{s}^3},
  \label{eq:subproblem}
\end{align}
where $\sigma > 0$ is a regularization parameter.
Then the iterate is updated as $x_{k+1} = x_k + s_k$.
The parameter $\sigma$ is adjusted by the backtracking procedure described later.

As an approximate solution to \cref{eq:subproblem}, we require a step $s \in \R^d$ satisfying the following first- and second-order optimality conditions:
\begin{align}
  \norm[\big]{g_k + H_k s + \sigma \norm*{s} s} &\leq \frac{\sigma}{12} \norm*{s}^2,
  \label{eq:subproblem_fo_condition}\\
  H_k + \frac{7}{6} \sigma \norm*{s} I &\succeq O.
  \label{eq:subproblem_so_condition}
\end{align}
Any optimal solution $s^* \in \R^d$ to \cref{eq:subproblem} satisfies the following conditions \citep[Theorem~3.1]{cartis2011adaptive}:
\begin{align}
  g_k + H_k s^* + \sigma \norm*{s^*} s^* = \0,\qquad
  H_k + \sigma \norm*{s^*} I \succeq O.
\end{align}
The requirements in \cref{eq:subproblem_fo_condition,eq:subproblem_so_condition} are relaxations of these exact conditions.
Similar inexact relaxations have been used in the literature (see, e.g., \citep{cartis2019universal,cartis2022evaluation}).

Methods for solving the cubic-regularized subproblem~\cref{eq:subproblem} have been studied extensively: examples include reductions to one-dimensional equations~\citep{nesterov2006cubic,gould2010solving,gao2022approximate}, Lanczos-type methods~\citep{carmon2018analysis,carmon2020first,cartis2011adaptive,gould2020error}, reformulations as eigenvalue problems~\citep{jia2022solving,lieder2020solving}, and accelerated projected gradient methods~\citep{jiang2021accelerated}.
Any of these methods can be used in the algorithm, provided that the returned approximate solution $s \in \R^d$ satisfies \cref{eq:subproblem_fo_condition,eq:subproblem_so_condition}.
We note that, once $g_k$ and $H_k$ have been computed, solving the subproblem requires no additional oracle calls, although it may incur arithmetic cost.

\paragraph{Backtracking for $\sigma$.}
To choose the regularization parameter $\sigma$, we employ a backtracking procedure.
Specifically, we repeatedly double $\sigma$ until the computed approximate solution $s \in \R^d$ to the subproblem satisfies the following two conditions:
\begin{align}
  f(x_k + s) - f(x_k) &\leq -\frac{\sigma}{6} \norm*{s}^3,
  \label{eq:accept_decrease_condition}\\
  \norm*{\nabla f(x_k + s)} &\leq \frac{19}{12} \sigma \norm*{s}^2.
  \label{eq:accept_grad_condition}
\end{align}
We denote the accepted pair $(\sigma, s)$ by $(\sigma_k, s_k)$.
Conditions of the same form as \cref{eq:accept_decrease_condition,eq:accept_grad_condition} can be found in \citep{birgin2017use,curtis2017trust} and \citep{cartis2019universal}, respectively.
Combining \cref{eq:accept_decrease_condition,eq:accept_grad_condition} gives
\begin{align}
  f(x_k + s) - f(x_k)
  \leq
  - \frac{c}{\sqrt{\sigma}} \norm*{\nabla f(x_k + s)}^{3/2},
\end{align}
for a constant $c > 0$; conditions of this weaker form are also used in the literature \citep{grapiglia2017regularized,grapiglia2020tensor}.
To accommodate generalized smoothness, we check the two conditions~\cref{eq:accept_decrease_condition,eq:accept_grad_condition} separately.

Once $s_k$ has been computed, we set $\sigma \gets \sigma / 2$ for the next iteration, as specified in Line~\ref{alg-line:sigma_halving} of \cref{alg:cubic_newton_backtracking}.
In the classical case $L_1 = 0$, this halving step is often introduced for better practical performance, although many complexity analyses do not rely on it.
In the generalized smoothness setting, however, this step plays an essential role in our analysis.
To understand why, let us compare the Taylor bound in~\cref{eq:function_taylor_bound} with the subproblem~\cref{eq:subproblem}.
For the subproblem objective to majorize $f(x_k + s) - f(x_k)$, the parameter $\sigma$ should be large enough to satisfy
\begin{align}
  \sigma
  \geq
  3 \tilde L(x_k, s) \frac{\sinh \alpha - \alpha}{\alpha^3}
  \geq
  \frac{\tilde L(x_k, s)}{2}
  =
  \frac{1}{2} \prn*{
    L_0 + L_1 \max \set*{\norm*{g_k},\, \norm*{g_k + H_k s}}
  }
  \label{eq:sigma_lower_bound_intuition}
\end{align}
where $\alpha = \sqrt{L_1} \norm*{s}$ and the second inequality follows from $\frac{\sinh t - t}{t^3} \geq \frac{1}{6}$.
Thus, when $L_1 > 0$, the backtracking procedure may need to increase $\sigma$ when $\norm*{g_k}$ is large.
Without the halving step, once $\sigma$ has been increased due to a large gradient norm, it may remain unnecessarily large, leading to a worse complexity bound.
See \cref{rem:sigma_halving} for a more detailed explanation of its effect on the complexity bound.

\subsection{Termination of backtracking}
This subsection establishes finite termination of the backtracking procedure in \cref{alg:cubic_newton_backtracking}.
In the case $L_1 = 0$, the standard Taylor remainder bounds show that the backtracking conditions~\cref{eq:accept_decrease_condition,eq:accept_grad_condition} are satisfied once $\sigma \geq L_0$.
For general $L_1 > 0$, the analogous lower bound on $\sigma$ is given in \cref{eq:sigma_lower_bound_intuition}.
The difficulty here is that this bound involves the trial step $s$, which is computed after $\sigma$ is fixed.
The following lemma gives a sufficient condition on $\sigma$ that is independent of the trial step.
Since $\sigma$ is doubled during backtracking, finite termination follows immediately from this condition.
\begin{lemma}
  \label{lem:sigma_acceptance_condition}
  Suppose that \cref{asm:our_hessian_gen_smoothness} holds and that
  \begin{align}
    \sigma \geq L_0 + \frac{12}{11} \prn*{
      L_1 \norm*{g_k} + \sqrt{L_1} \mu_k
    }.
    \label{eq:sigma_acceptance_bound}
  \end{align}
  If $s \in \R^d$ satisfies \cref{eq:subproblem_fo_condition,eq:subproblem_so_condition}, then \cref{eq:accept_decrease_condition,eq:accept_grad_condition} hold.
\end{lemma}
In the case $L_1 = 0$, the bound~\cref{eq:sigma_acceptance_bound} reduces to $\sigma \geq L_0$, which matches the standard result for functions with Lipschitz-continuous Hessians (see, e.g., \citep[Eq.~(5.7)]{cartis2011adaptive}).
Thus, the lemma extends this standard guarantee to generalized smoothness.

The rest of this section is devoted to proving \cref{lem:sigma_acceptance_condition}.
For $s \in \R^d$, define
\begin{align}
  G_k(s) \coloneqq \max \set*{\norm*{g_k}, \norm*{g_k + H_k s}}.
  \label{eq:def_Gk_s}
\end{align}
Then $\tilde L(x_k, s) = L_0 + L_1 G_k(s)$.
The following lemma bounds $G_k(s)$ when the trial step $s \in \R^d$ satisfies an approximate first-order condition for the subproblem.
We state the result for a general parameter $\theta \in [0, 1)$ so that the argument is not obscured by the particular constant $\frac{1}{12}$ in \cref{eq:subproblem_fo_condition}.

\begin{lemma}
  \label{lem:Gk_s_bound}
  Let $\theta \in [0, 1)$ and $\sigma > 0$.
  Suppose that $s \in \R^d$ satisfies
  \begin{align}
    \norm[\big]{g_k + H_k s + \sigma \norm*{s} s} \leq \theta \sigma \norm*{s}^2.
    \label{eq:subproblem_theta_fo_condition}
  \end{align}
  Then, the following hold:
  \begin{align}
    (1 - \theta) \sigma \norm*{s}^2
    &\leq
    \norm*{g_k} + \mu_k \norm*{s},
    \label{eq:sigma_s2_lower_bound}\\
    G_k(s) + \theta \sigma \norm*{s}^2
    &\leq
    \frac{\norm*{g_k} + \mu_k \norm*{s}}{1 - \theta}.
    \label{eq:Gk_s_theta_bound}
  \end{align}
\end{lemma}

\begin{proof}
  The case $s = \0$ is trivial, so we assume that $s \neq \0$.
  We first derive upper and lower bounds on $\inner*{g_k + H_k s}{s}$.
  We have
  \begin{align}
    \inner*{g_k + H_k s}{s}
    &=
    \inner[\big]{g_k + H_k s + \sigma \norm*{s} s}{s}
    - \sigma \norm*{s}^3\\
    &\leq
    \norm[\big]{g_k + H_k s + \sigma \norm*{s} s} \norm*{s}
    - \sigma \norm*{s}^3
    \leq
    - (1 - \theta) \sigma \norm*{s}^3,
  \end{align}
  where the last inequality uses \cref{eq:subproblem_theta_fo_condition}.
  We also have
  \begin{align}
    \inner*{g_k + H_k s}{s}
    &=
    \inner*{g_k}{s} + \inner*{H_k s}{s}
    \geq
    - \norm*{g_k} \norm*{s} - \mu_k \norm*{s}^2
    \label{eq:gk_Hk_s_inner_lower_bound}
  \end{align}
  where the inequality follows from Cauchy--Schwarz and the definition of $\mu_k$ in \cref{eq:def_g_H_mu}.
  Combining these two bounds and dividing by $\norm*{s}$ proves \cref{eq:sigma_s2_lower_bound}.
  Using \cref{eq:sigma_s2_lower_bound} gives
  \begin{align}
    \norm*{g_k} + \theta \sigma \norm*{s}^2
    \leq
    \norm*{g_k} + \frac{\theta}{1 - \theta} \prn[\Big]{\norm*{g_k} + \mu_k \norm*{s}}
    =
    \frac{\norm*{g_k} + \theta \mu_k \norm*{s}}{1 - \theta}.
    \label{eq:gk_theta_sigma_s2_bound}
  \end{align}

  Next, squaring \cref{eq:subproblem_theta_fo_condition} gives
  \begin{align}
    \norm*{g_k + H_k s}^2
    + 2 \sigma \norm*{s} \inner*{g_k + H_k s}{s}
    + \sigma^2 \norm*{s}^4
    \leq
    \theta^2 \sigma^2 \norm*{s}^4.
  \end{align}
  Since $2 ab \leq a^2 + b^2$ for $a, b \in \R$, we have
  \begin{align}
    2 \prn*{1 - \theta} \sigma \norm*{s}^2 \norm*{g_k + H_k s}
    \leq
    \prn*{1 - \theta}^2 \sigma^2 \norm*{s}^4
    + \norm*{g_k + H_k s}^2.
  \end{align}
  Adding these two inequalities and rearranging terms gives
  \begin{align}
    2 \prn*{1 - \theta} \sigma \norm*{s}^2
    \prn*{
      \norm*{g_k + H_k s} + \theta \sigma \norm*{s}^2
    }
    &\leq
    - 2 \sigma \norm*{s} \inner*{g_k + H_k s}{s}.
  \end{align}
  Dividing both sides by $2 \prn*{1 - \theta} \sigma \norm*{s}^2$ and using \cref{eq:gk_Hk_s_inner_lower_bound} yields
  \begin{align}
    \norm*{g_k + H_k s} + \theta \sigma \norm*{s}^2
    &\leq
    - \frac{\inner*{g_k + H_k s}{s}}{\prn*{1 - \theta} \norm*{s}}
    \leq
    \frac{\norm*{g_k} + \mu_k \norm*{s}}{1 - \theta}.
    \label{eq:gk_Hk_s_theta_bound}
  \end{align}
  Combining \cref{eq:gk_theta_sigma_s2_bound,eq:gk_Hk_s_theta_bound} with the definition of $G_k(s)$ and $\theta < 1$ proves \cref{eq:Gk_s_theta_bound}.
\end{proof}

The following lemma shows that, once $\sigma$ is sufficiently large, the approximate first-order condition for the subproblem yields an upper bound on $\tilde L(x_k, s)$.
This quantity appears in \cref{thm:taylor_equivalence}, and controlling it is the key step in proving \cref{lem:sigma_acceptance_condition}.

\begin{lemma}
  \label{lem:M0_M1_Gk_bound}
  Suppose that \cref{asm:our_hessian_gen_smoothness} holds and that $\sigma > 0$ satisfies \cref{eq:sigma_acceptance_bound}.
  For every $s \in \R^d$ satisfying \cref{eq:subproblem_fo_condition}, we have
  \begin{align}
    \tilde L(x_k, s)
    \leq
    \sigma \prn*{
      1 - \frac{L_1}{12} \norm*{s}^2
    }.
  \end{align}
\end{lemma}
\begin{proof}
  Since $s \in \R^d$ satisfies \cref{eq:subproblem_fo_condition}, we can apply \cref{lem:Gk_s_bound} with $\theta = \frac{1}{12}$ to obtain
  \begin{align}
    \frac{11}{12} \sigma \norm*{s}^2
    &\leq
    \norm*{g_k} + \mu_k \norm*{s},\\
    G_k(s) + \frac{\sigma}{12} \norm*{s}^2
    &\leq
    \frac{12}{11} \prn[\Big]{ \norm*{g_k} + \mu_k \norm*{s} }.
    \label{eq:Gk_s_theta_1_12_bound}
  \end{align}
  Let $\alpha \coloneqq \sqrt{L_1} \norm*{s}$. Suppose, for a contradiction, that $\alpha > 1$.
  Multiplying the first inequality by $L_1$ gives
  \begin{align}
    \frac{11}{12} \sigma \alpha^2
    &\leq
    L_1 \norm*{g_k} + \sqrt{L_1} \mu_k \alpha
    \leq
    \prn*{L_1 \norm*{g_k} + \sqrt{L_1} \mu_k} \alpha
    \leq
    \frac{11}{12} \sigma \alpha,
  \end{align}
  where the second inequality follows from $\alpha > 1$, and the last inequality uses \cref{eq:sigma_acceptance_bound}.
  This implies $\alpha \leq 1$, contradicting $\alpha > 1$.

  Using $\alpha = \sqrt{L_1}\norm*{s} \leq 1$ and \cref{eq:sigma_acceptance_bound}, we obtain
  \begin{align}
    L_1 \prn[\Big]{ \norm*{g_k} + \mu_k \norm*{s} }
    \leq
    L_1 \norm*{g_k} + \sqrt{L_1} \mu_k
    \leq
    \frac{11}{12} \prn*{\sigma - L_0}.
  \end{align}
  Multiplying \cref{eq:Gk_s_theta_1_12_bound} by $L_1$ and applying the above bound yields
  \begin{align}
    L_1 \prn*{ G_k(s) + \frac{\sigma}{12} \norm*{s}^2 }
    \leq
    \frac{12}{11} L_1 \prn[\Big]{ \norm*{g_k} + \mu_k \norm*{s} }
    \leq
    \sigma - L_0.
  \end{align}
  Rearranging terms completes the proof.
\end{proof}

The inequalities in \cref{thm:taylor_equivalence} involve not only $\tilde L(x_k, s)$ but also the hyperbolic functions.
The following auxiliary lemma provides the bounds needed to control these terms.

\begin{lemma}
  \label{lem:hyperbolic_bounds}
  For any $t \geq 0$, we have
  \begin{align}
    \prn*{1 - \frac{t^2}{12}} \frac{\cosh t - 1}{t^2} \leq \frac{1}{2}, \qquad
    \prn*{1 - \frac{t^2}{12}} \frac{\sinh t - t}{t^3} \leq \frac{1}{6}.
  \end{align}
\end{lemma}
\begin{proof}
  The power series expansions give
  \begin{align}
    \cosh t - 1
    = \sum_{m=1}^{\infty} \frac{t^{2 m}}{(2 m)!},\qquad
    \sinh t - t
    = \sum_{m=1}^{\infty} \frac{t^{2 m+1}}{\prn*{2 m + 1}!},
  \end{align}
  and hence
  \begin{align}
    \prn*{1 - \frac{t^2}{12}} \frac{\cosh t - 1}{t^2}
    &=
    \frac{1}{2}
    + \sum_{m=1}^{\infty} \frac{1}{(2 m)!} \prn*{
      \frac{1}{\prn*{2 m + 1}\prn*{2 m + 2}} - \frac{1}{12}
    } t^{2m},\\
    \prn*{1 - \frac{t^2}{12}} \frac{\sinh t - t}{t^3}
    &=
    \frac{1}{6}
    + \sum_{m=1}^{\infty} \frac{1}{\prn*{2 m + 1}!} \prn*{
      \frac{1}{\prn*{2 m + 2}\prn*{2 m + 3}} - \frac{1}{12}
    } t^{2m}.
  \end{align}
  Since all coefficients in the sums are nonpositive, the desired bounds follow.
\end{proof}

We are now ready to prove \cref{lem:sigma_acceptance_condition}.
The proof combines \cref{lem:M0_M1_Gk_bound,lem:hyperbolic_bounds}.

\begin{proof}[Proof of \cref{lem:sigma_acceptance_condition}]
  If $s = \0$, then \cref{eq:accept_decrease_condition,eq:accept_grad_condition} are immediate.
  We assume $s \neq \0$ and set $\alpha \coloneqq \sqrt{L_1} \norm*{s}$.
  Combining \cref{lem:M0_M1_Gk_bound,lem:hyperbolic_bounds} yields
  \begin{align}
    \tilde L(x_k, s) \frac{\cosh \alpha - 1}{\alpha^2}
    &\leq
    \sigma \prn*{1 - \frac{\alpha^2}{12}} \frac{\cosh \alpha - 1}{\alpha^2}
    \leq
    \frac{\sigma}{2},\\
    \tilde L(x_k, s) \frac{\sinh \alpha - \alpha}{\alpha^3}
    &\leq
    \sigma \prn*{1 - \frac{\alpha^2}{12}} \frac{\sinh \alpha - \alpha}{\alpha^3}
    \leq
    \frac{\sigma}{6}.
  \end{align}
  Plugging these bounds into \cref{eq:grad_taylor_G_bound,eq:function_taylor_bound} gives
  \begin{align}
    \norm*{\nabla f(x_k + s) - g_k - H_k s}
    &\leq
    \frac{\sigma}{2} \norm*{s}^2,
    \label{eq:accept_grad_remainder_bound}\\
    f(x_k + s) - f(x_k) - \inner*{g_k}{s} - \frac{1}{2} \inner*{H_k s}{s}
    &\leq
    \frac{\sigma}{6} \norm*{s}^3.
    \label{eq:accept_taylor_remainder_bound}
  \end{align}
  Using the triangle inequality, we have
  \begin{align}
    \norm*{\nabla f(x_k + s)}
    \leq
    \norm*{\nabla f(x_k + s) - g_k - H_k s}
    + \norm[\big]{g_k + H_k s + \sigma \norm*{s} s}
    + \sigma \norm*{s}^2.
  \end{align}
  Plugging \cref{eq:accept_grad_remainder_bound,eq:subproblem_fo_condition} into this bound proves \cref{eq:accept_grad_condition}.
  We also have
  \begin{alignat}{2}
    \inner*{g_k}{s} + \frac{1}{2} \inner*{H_k s}{s}
    &=
    \inner[\big]{g_k + H_k s + \sigma \norm*{s} s}{s}
    - \frac{1}{2} \inner*{H_k s}{s}
    - \sigma \norm*{s}^3\\
    &\leq
    \frac{\sigma}{12} \norm*{s}^3
    - \frac{1}{2} \inner*{H_k s}{s}
    - \sigma \norm*{s}^3
    &\quad&\by{\cref{eq:subproblem_fo_condition}}\\
    &\leq
    \frac{\sigma}{12} \norm*{s}^3
    + \frac{7}{12} \sigma \norm*{s}^3
    - \sigma \norm*{s}^3
    =
    - \frac{\sigma}{3} \norm*{s}^3.
    &\quad&\by{\cref{eq:subproblem_so_condition}}
  \end{alignat}
  Adding this bound to \cref{eq:accept_taylor_remainder_bound} proves \cref{eq:accept_decrease_condition}.
\end{proof}

\section{Complexity analysis}
This section analyzes the oracle complexity of \cref{alg:cubic_newton_backtracking} under \cref{asm:main}.

We first note several inequalities that will be used in the analysis.
The backtracking conditions~\cref{eq:accept_decrease_condition,eq:accept_grad_condition} ensure that
\begin{align}
  \frac{\sigma_k}{6} \norm*{s_k}^3
  &\leq
  f(x_k) - f(x_{k+1}),
  \label{eq:func_decrease_sk_bound}
  \\
  \norm*{g_{k+1}} &\leq \frac{19}{12} \sigma_k \norm*{s_k}^2.
  \label{eq:g_k_plus_1_sigmak_sk_bound}
\end{align}
By the definition of $\mu_k$ in \cref{eq:def_g_H_mu}, condition~\cref{eq:subproblem_so_condition} implies
\begin{align}
  \mu_k
  &\leq
  \frac{7}{6} \sigma_k \norm*{s_k}.
  \label{eq:mu_k_sigmak_sk_bound}
\end{align}
Summing \cref{eq:func_decrease_sk_bound} over $k$ gives
\begin{align}
  \sum_{i=0}^{k-1} \sigma_i \norm*{s_i}^3 \leq 6 \prn*{f(x_0) - f(x_k)} \leq 6 \Delta.
  \label{eq:sigmai_si_cubic_sum_bound}
\end{align}

\subsection{Upper bounds on the regularization parameter}
To derive the complexity bound from the inequalities above, we also need an upper bound on the accepted regularization parameter $\sigma_k$.
When $L_1 = 0$, this step is straightforward; \cref{lem:sigma_acceptance_condition} immediately implies $\sigma_k \leq \max \set*{\sigmainit,\, 2 L_0}$.
In the case of primary interest, $L_1 > 0$, the argument is more delicate because the sufficient condition in \cref{lem:sigma_acceptance_condition} is no longer uniform in $k$; it contains the additional $k$-dependent terms $\norm*{g_k}$ and $\mu_k$.
To handle this dependence, we first establish an upper bound on $\sigma_0$ and a recursion for $\sigma_k$.

\begin{lemma}
  Suppose that \cref{asm:our_hessian_gen_smoothness} holds.
  Then we have
  \begin{align}
    \sigma_0
    \leq
    \max \set*{
      \sigmainit,\,
      4 L_0 + 4 L_1 \norm*{g_0} + 24 L_1^{3/2} \sigma_0 \norm*{s_0}^3
    }
    \label{eq:sigma0_upper_bound}
  \end{align}
  Moreover, for all $k \geq 1$,
  \begin{align}
    \sigma_k
    \leq
    \frac{\sigma_{k-1}}{2} + 4 L_0 + 8 L_1^{3/2} \prn*{
      \sigma_{k-1} \norm*{s_{k-1}}^3
      + 3 \sigma_k \norm*{s_k}^3
    }.
    \label{eq:sigmak_recursion_bound}
  \end{align}
\end{lemma}

\begin{proof}
  The weighted AM--GM inequality $a^{2/3} b^{1/3} \leq \frac{2}{3} a + \frac{1}{3} b$ gives, for all $t \geq 0$,
  \begin{align}
    t
    &=
    \prn*{\frac{1}{4}}^{2/3} \prn*{16 t^3}^{1/3}
    \leq
    \frac{1}{6} + \frac{16}{3} t^3,
    \label{eq:t_young_linear_bound}\\
    t^2
    &=
    \prn*{2 t^3}^{2/3} \prn*{\frac{1}{4}}^{1/3}
    \leq
    \frac{4}{3} t^3 + \frac{1}{12}.
    \label{eq:t_square_young_bound}
  \end{align}
  Let $\alpha_k \coloneqq \sqrt{L_1} \norm*{s_k}$ throughout this proof.

  We first prove \cref{eq:sigma0_upper_bound}.
  Since the first trial value of $\sigma$ in iteration $0$ is $\sigmainit$, \cref{lem:sigma_acceptance_condition} implies that $\sigma_0 = \sigmainit$ or
  \begin{align}
    \sigma_0
    &\leq
    2 L_0 + \frac{24}{11} \prn*{L_1 \norm*{g_0} + \sqrt{L_1} \mu_0}.
    \label{eq:sigma0_backtracking_bound}
  \end{align}
  The case $\sigma_0 = \sigmainit$ is immediate, so assume now that \cref{eq:sigma0_backtracking_bound} holds.
  Plugging \cref{eq:mu_k_sigmak_sk_bound} into \cref{eq:sigma0_backtracking_bound} gives
  \begin{align}
    \sigma_0
    &\leq
    2 L_0 + \frac{24}{11} L_1 \norm*{g_0} + \frac{28}{11} \sigma_0 \alpha_0
    \leq
    2 L_0 + \frac{24}{11} L_1 \norm*{g_0} + \frac{28}{11} \sigma_0 \prn*{
      \frac{1}{6} + \frac{16}{3} \alpha_0^3
    },
  \end{align}
  where the second inequality uses \cref{eq:t_young_linear_bound} with $t=\alpha_0$.
  Collecting the $\sigma_0$ terms on the left and simplifying yields
  \begin{align}
    \sigma_0
    \leq
    \frac{66}{19} L_0 + \frac{72}{19} L_1 \norm*{g_0} + \frac{448}{19} \sigma_0 \alpha_0^3
    &\leq
    4 L_0 + 4 L_1 \norm*{g_0} + 24 \sigma_0 \alpha_0^3,
  \end{align}
  where we have used $\frac{66}{19} \leq 4$, $\frac{72}{19} \leq 4$, and $\frac{448}{19} \leq 24$.
  Substituting $\alpha_0 = \sqrt{L_1} \norm*{s_0}$ proves \cref{eq:sigma0_upper_bound}.

  We next prove \cref{eq:sigmak_recursion_bound} in a similar manner.
  Let $k \geq 1$.
  Since the first trial value of $\sigma$ in iteration $k$ is $\sigma_{k-1} / 2$, \cref{lem:sigma_acceptance_condition} implies that $\sigma_k = \sigma_{k-1} / 2$ or
  \begin{align}
    \sigma_k
    \leq
    2 L_0 + \frac{24}{11} \prn*{L_1 \norm*{g_k} + \sqrt{L_1} \mu_k}.
    \label{eq:sigmak_backtracking_bound}
  \end{align}
  The case $\sigma_k = \sigma_{k-1} / 2$ is immediate, so assume now that \cref{eq:sigmak_backtracking_bound} holds.
  Plugging \cref{eq:g_k_plus_1_sigmak_sk_bound,eq:mu_k_sigmak_sk_bound} into \cref{eq:sigmak_backtracking_bound} gives
  \begin{alignat}{2}
    \sigma_k
    &\leq
    2 L_0 + \frac{24}{11} \prn*{
      \frac{19}{12} \sigma_{k-1} \alpha_{k-1}^2
      + \frac{7}{6} \sigma_k \alpha_k
    }\\
    &\leq
    2 L_0 + \frac{24}{11} \prn*{
      \frac{19}{12} \sigma_{k-1} \prn*{\frac{1}{12} + \frac{4}{3} \alpha_{k-1}^3}
      + \frac{7}{6} \sigma_k \prn*{\frac{1}{6} + \frac{16}{3} \alpha_k^3}
    }\\
    &=
    2 L_0 + \frac{19}{66} \sigma_{k-1} + \frac{8 \cdot 19}{33} \sigma_{k-1} \alpha_{k-1}^3 + \frac{14}{33} \sigma_k + \frac{448}{33} \sigma_k \alpha_k^3,
  \end{alignat}
  where the second inequality uses \cref{eq:t_square_young_bound,eq:t_young_linear_bound}.
  Collecting the $\sigma_k$ terms on the left and simplifying yields
  \begin{align}
    \sigma_k
    \leq
    \frac{\sigma_{k-1}}{2} + \frac{66}{19} L_0 + 8 \sigma_{k-1} \alpha_{k-1}^3 + \frac{448}{19} \sigma_k \alpha_k^3.
  \end{align}
  Using $\frac{66}{19} \leq 4$ and $\frac{448}{19} \leq 24$ and then substituting $\alpha_k = \sqrt{L_1} \norm*{s_k}$ proves \cref{eq:sigmak_recursion_bound}.
\end{proof}

Notice that the upper bound on $\sigma_0$ in \cref{eq:sigma0_upper_bound} contains terms, such as $\norm*{g_0}$, that did not appear in the case $L_1 = 0$.
On the other hand, the right-hand side of \cref{eq:sigmak_recursion_bound} contains the factor $\frac{1}{2}$ multiplying $\sigma_{i-1}$.
Thus, iterating this inequality shows that the influence of $\norm*{g_0}$ and $\sigmainit$ on $\sigma_i$ decays geometrically in $i$.
More precisely, define the index $k_0 \in \N$ by
\begin{align}
  k_0
  \coloneqq
  \max \set[\bigg]{1, \ceil[\bigg]{\log_2 \prn[\bigg]{\frac{\bar{\sigma}_{\mathrm{init}}}{L_0 + L_1^{3/2} \Delta}}}},
  \quad\text{where}\quad
  \bar{\sigma}_{\mathrm{init}}
  \coloneqq
  \max \set*{\sigmainit, 4 L_1 \norm*{g_0}}.
  \label{eq:def_initial_phase_k0_sigma_init_bar}
\end{align}
For iterations $k \geq k_0$, the influence of $\norm*{g_0}$ and $\sigmainit$ becomes negligible in the subsequent complexity analysis, as the following lemma shows.
This definition immediately gives
\begin{align}
  2^{-k_0} \bar{\sigma}_{\mathrm{init}} \leq L_0 + L_1^{3/2} \Delta.
  \label{eq:initial_phase_k0_condition}
\end{align}

\begin{lemma}
  \label{lem:sigma_post_sum_bound}
  Suppose that \cref{asm:main} holds and let $k_0$ be defined by \cref{eq:def_initial_phase_k0_sigma_init_bar}.
  For all $i \geq k_0$, we have
  \begin{align}
    \sigma_i \leq 9 L_0 + 145 L_1^{3/2} \Delta.
    \label{eq:sigmai_post_upper_bound}
  \end{align}
  Moreover, for all $k \geq k_0$,
  \begin{align}
    \sum_{i = k_0}^{k} \sigma_i \leq \prn[\big]{8 \prn*{k - k_0} + 10} L_0 + 386 L_1^{3/2} \Delta.
    \label{eq:sigmai_post_sum_bound}
  \end{align}
\end{lemma}

\begin{proof}
  To simplify notation, let
  \begin{align}
    \Phi_i \coloneqq \sigma_i - 8 L_0 - 24 L_1^{3/2} \sigma_i \norm*{s_i}^3.
    \label{eq:def_Phi_i}
  \end{align}
  Applying \cref{eq:sigma0_upper_bound} to the definition of $\Phi_0$ gives $\Phi_0 \leq \max \set*{\sigmainit, 4 L_1 \norm*{g_0} - 4 L_0} \leq \bar{\sigma}_{\mathrm{init}}$, where $\bar{\sigma}_{\mathrm{init}}$ is defined in \cref{eq:def_initial_phase_k0_sigma_init_bar}.
  Moreover, \cref{eq:sigmak_recursion_bound} can be rewritten as
  \begin{align}
    \Phi_i
    \leq
    \frac{\Phi_{i-1}}{2} + 20 L_1^{3/2} \sigma_{i-1} \norm*{s_{i-1}}^3.
  \end{align}
  Solving this recursion and using $\Phi_0 \leq \bar{\sigma}_{\mathrm{init}}$ gives
  \begin{align}
    \Phi_i
    \leq
    2^{-i} \bar{\sigma}_{\mathrm{init}}
    + 20 L_1^{3/2} \sum_{j=0}^{i-1} 2^{j-i+1} \sigma_j \norm*{s_j}^3.
    \label{eq:Phi_i_unrolled_bound}
  \end{align}
  Plugging \cref{eq:def_Phi_i} into this bound yields
  \begin{align}
    \sigma_i
    &\leq
    2^{-i} \bar{\sigma}_{\mathrm{init}}
    + 8 L_0
    + 24 L_1^{3/2} \sigma_i \norm*{s_i}^3
    + 20 L_1^{3/2} \sum_{j=0}^{i-1} 2^{j-i+1}  \sigma_j \norm*{s_j}^3\\
    &\leq
    2^{-k_0} \bar{\sigma}_{\mathrm{init}}
    + 8 L_0
    + 24 L_1^{3/2} \sum_{j=0}^i \sigma_j \norm*{s_j}^3,
  \end{align}
  where, for the second inequality, we have used $i \geq k_0$ and $2^{j-i+1} \leq 1$ for $0 \leq j \leq i-1$.
  Plugging \cref{eq:sigmai_si_cubic_sum_bound,eq:initial_phase_k0_condition} into this bound proves \cref{eq:sigmai_post_upper_bound}.

  To prove \cref{eq:sigmai_post_sum_bound}, summing \cref{eq:Phi_i_unrolled_bound} over $i = k_0, \ldots, k$ yields
  \begin{align}
    \sum_{i = k_0}^{k} \Phi_i
    &\leq
    \sum_{i = k_0}^{k} 2^{-i} \bar{\sigma}_{\mathrm{init}}
    + 20 L_1^{3/2} \sum_{i = k_0}^{k} \sum_{j=0}^{i-1} 2^{j-i+1} \sigma_j \norm*{s_j}^3\\
    &\leq
    2^{1 - k_0} \bar{\sigma}_{\mathrm{init}}
    + 40 L_1^{3/2} \sum_{j=0}^{k-1} \sigma_j \norm*{s_j}^3,
    \label{eq:Phi_i_sum_bound}
  \end{align}
  where the double sum is bounded by exchanging the order of summation:
  \begin{align}
    \sum_{i = k_0}^{k} \sum_{j=0}^{i-1} 2^{j-i+1} \sigma_j \norm*{s_j}^3
    &=
    \sum_{j=0}^{k-1} \sigma_j \norm*{s_j}^3
    \prn[\Bigg]{ \sum_{i = \max\set*{k_0, j+1}}^{k} 2^{j-i+1} }\\
    &\leq
    \sum_{j=0}^{k-1} \sigma_j \norm*{s_j}^3
    \prn[\Bigg]{ \sum_{i = j+1}^\infty 2^{j-i+1} }
    =
    2 \sum_{j=0}^{k-1} \sigma_j \norm*{s_j}^3.
  \end{align}
  Applying \cref{eq:initial_phase_k0_condition,eq:sigmai_si_cubic_sum_bound} to \cref{eq:Phi_i_sum_bound} gives
  \begin{align}
    \sum_{i = k_0}^{k} \Phi_i
    \leq
    2 \prn*{
      L_0 + L_1^{3/2} \Delta
    }
    + 40 L_1^{3/2} \cdot (6 \Delta).
  \end{align}
  Also, summing \cref{eq:def_Phi_i} over $i = k_0, \ldots, k$ and then using \cref{eq:sigmai_si_cubic_sum_bound} gives
  \begin{align}
    \sum_{i = k_0}^{k} \sigma_i
    &\leq
    \sum_{i = k_0}^{k} \Phi_i
    + 8 L_0 \prn*{k - k_0 + 1}
    + 24 L_1^{3/2} \cdot (6 \Delta).
  \end{align}
  Combining the last two bounds yields
  \begin{align}
    \sum_{i = k_0}^{k} \sigma_i
    &\leq
    2 \prn*{L_0 + L_1^{3/2} \Delta}
    + 40 L_1^{3/2} \cdot (6 \Delta)
    + 8 L_0 \prn*{k - k_0 + 1}
    + 24 L_1^{3/2} \cdot (6 \Delta)\\
    &=
    \prn*{8 \prn*{k - k_0} + 10} L_0 + 386 L_1^{3/2} \Delta,
  \end{align}
  which completes the proof of \cref{eq:sigmai_post_sum_bound}.
\end{proof}

\begin{remark}
  \label{rem:sigma_halving}
  The halving step $\sigma \gets \sigma_k / 2$ is essential in \cref{lem:sigma_post_sum_bound}.
  It is the source of the geometric decay in \cref{eq:Phi_i_unrolled_bound}, which is used to prove \cref{eq:sigmai_post_sum_bound}.
  Without this decay, the double sum in \cref{eq:Phi_i_sum_bound} would acquire an extra factor of $\O(k)$.
  Consequently, the $L_1^{3/2} \Delta$ term in \cref{eq:sigmai_post_sum_bound} would acquire the same extra factor, which worsens the complexity bound.
\end{remark}

\subsection{Oracle complexity}
Combining \cref{eq:g_k_plus_1_sigmak_sk_bound,eq:mu_k_sigmak_sk_bound,eq:sigmai_si_cubic_sum_bound} with \cref{lem:sigma_post_sum_bound}, we obtain the following iteration complexity bound.

\begin{theorem}
  Suppose that \cref{asm:main} holds.
  Let $\epsilon, \delta > 0$.
  Define
  \begin{align}
    m
    \coloneqq
    \ceil*{
      \Delta \max \set*{
        \frac{\sqrt{L_0}}{(\epsilon / 56)^{3/2}},\,
        \frac{\sqrt{L_1}}{\epsilon / 156},\,
        \frac{L_0^2}{(\delta / 60)^3},\,
        \frac{L_1}{\delta / 452}
      }
    }
    \geq 1
    \label{eq:def_m}
  \end{align}
  and $K \coloneqq k_0 + m$, where $k_0$ is defined by \cref{eq:def_initial_phase_k0_sigma_init_bar}.
  Then there exists an index $k \in \set*{k_0 + 1, \ldots, K}$ such that $x_k$ is an $(\epsilon, \delta)$-stationary point.
\end{theorem}

\begin{proof}
  Combining \cref{eq:g_k_plus_1_sigmak_sk_bound,eq:mu_k_sigmak_sk_bound} with H\"{o}lder's inequality gives
  \begin{align}
    \sum_{i = k_0 + 1}^{K} \norm*{g_i}
    &\leq
    \frac{19}{12}
    \sum_{i = k_0}^{K - 1} \prn*{\sigma_i \norm*{s_i}^3}^{2/3} \sigma_i^{1/3}
    \leq
    \frac{19}{12}
    \prn[\Bigg]{\sum_{i = k_0}^{K - 1} \sigma_i \norm*{s_i}^3}^{2/3}
    \prn[\Bigg]{\sum_{i = k_0}^{K - 1} \sigma_i}^{1/3},\\
    \sum_{i = k_0 + 1}^{K} \mu_i
    &\leq
    \frac{7}{6}
    \sum_{i = k_0 + 1}^{K} \prn*{\sigma_i \norm*{s_i}^3}^{1/3} \sigma_i^{2/3}
    \leq
    \frac{7}{6}
    \prn[\Bigg]{\sum_{i = k_0 + 1}^{K} \sigma_i \norm*{s_i}^3}^{1/3}
    \prn[\Bigg]{\sum_{i = k_0 + 1}^{K} \sigma_i}^{2/3}.
  \end{align}
  Plugging \cref{eq:sigmai_si_cubic_sum_bound,eq:sigmai_post_sum_bound} into these bounds and using $8 m + 10 \leq 18 m$ gives
  \begin{align}
    \sum_{i = k_0 + 1}^{K} \norm*{g_i}
    &\leq
    \frac{19}{12} \prn*{6 \Delta}^{2/3}
    \prn*{18 m L_0 + 386 L_1^{3/2} \Delta}^{1/3},\\
    \sum_{i = k_0 + 1}^{K} \mu_i
    &\leq
    \frac{7}{6} \prn*{6 \Delta}^{1/3}
    \prn*{18 m L_0 + 386 L_1^{3/2} \Delta}^{2/3}.
  \end{align}
  Using these bounds together with $(a + b)^p \leq a^p + b^p$ for $a, b \geq 0$ and $p \in (0, 1]$ yields
  \begin{align}
    \frac{1}{m} \sum_{i = k_0 + 1}^{K} \prn*{
      \frac{\norm*{g_i}}{\epsilon} + \frac{\mu_i}{\delta}
    }
    &\leq
    \frac{19}{12 \epsilon m} \prn*{6 \Delta}^{2/3}
    \prn*{
      \prn*{ 18 m L_0 }^{1/3}
      + \prn*{ 386 L_1^{3/2} \Delta }^{1/3}
    }\\
    &\qquad
    + \frac{7}{6 \delta m} \prn*{6 \Delta}^{1/3}
    \prn*{
      \prn*{ 18 m L_0 }^{2/3}
      + \prn*{ 386 L_1^{3/2} \Delta }^{2/3}
    }\\
    &\leq
    14 \frac{\Delta^{2/3} L_0^{1/3}}{\epsilon m^{2/3}}
    + 39 \frac{\Delta L_1^{1/2}}{\epsilon m}
    + 15 \frac{\Delta^{1/3} L_0^{2/3}}{\delta m^{1/3}}
    + 113 \frac{\Delta L_1}{\delta m}.
  \end{align}
  To ensure that each of the last four terms is at most $1/4$, it suffices to choose $m$ as in \cref{eq:def_m}.
  Then it follows that $\frac{1}{m} \sum_{i = k_0 + 1}^{K} \prn[\big]{ \frac{\norm*{g_i}}{\epsilon} + \frac{\mu_i}{\delta} } \leq 1$.
  Hence, there exists $i \in \set*{k_0 + 1, \ldots, K}$ such that $\frac{\norm*{g_i}}{\epsilon} + \frac{\mu_i}{\delta} \leq 1$, which implies that $\norm*{g_i} \leq \epsilon$ and $\mu_i \leq \delta$.
\end{proof}

The above theorem bounds the number of outer iterations $K$ and hence the oracle calls used to compute $g_k$ and $H_k$ in Line~\ref{alg-line:compute_g_H} of \cref{alg:cubic_newton_backtracking}.
It remains to count the oracle calls used to check conditions~\cref{eq:accept_decrease_condition,eq:accept_grad_condition} in Line~\ref{alg-line:check_conditions}.
Combining these counts yields the following oracle complexity bound.

\begin{corollary}
  \label{cor:epsilon_delta_oracle_complexity}
  Suppose that \cref{asm:main} holds.
  \Cref{alg:cubic_newton_backtracking} finds an $(\epsilon, \delta)$-stationary point after at most
  \begin{align}
    \O \prn*{
      \Delta \prn*{
        \frac{\sqrt{L_0}}{\epsilon^{3/2}}
        + \frac{\sqrt{L_1}}{\epsilon}
        + \frac{L_0^2}{\delta^3}
        + \frac{L_1}{\delta}
      }
      + k_0
      + \brk[\bigg]{
        \log \frac{L_0 + L_1^{3/2} \Delta}{\sigmainit}
      }_+
    }
  \end{align}
  oracle calls, where $k_0$ is defined by \cref{eq:def_initial_phase_k0_sigma_init_bar}.
  The big-$\O$ notation hides only universal constants.
\end{corollary}

\begin{proof}
  Let $K \coloneqq k_0 + m$, where $k_0$ and $m$ are defined by \cref{eq:def_initial_phase_k0_sigma_init_bar,eq:def_m}, respectively.
  Exactly one trial is accepted in Line~\ref{alg-line:check_conditions} during each of the $K$ outer iterations, so it remains to count the rejected trials.
  Since the backtracking loop in iteration $i \geq 1$ starts with $\sigma = \sigma_{i-1} / 2$, the total number of rejected trials over the $K$ iterations is
  \begin{align}
    \log_2 \frac{\sigma_0}{\sigmainit}
    + \sum_{i=1}^{K-1} \log_2 \frac{\sigma_i}{\sigma_{i-1} / 2}
    =
    \log_2 \frac{\sigma_0}{\sigmainit}
    + \sum_{i=1}^{K-1} \prn*{1 + \log_2 \frac{\sigma_i}{\sigma_{i-1}}}
    =
    (K - 1) + \log_2 \frac{\sigma_{K-1}}{\sigmainit}.
  \end{align}
  Since $K - 1 = k_0 + m - 1 \geq k_0$, we can apply \cref{eq:sigmai_post_upper_bound} to bound $\sigma_{K-1}$.
  Hence, the total number of oracle calls is at most
  \begin{align}
    \O \prn*{
      K + \brk[\bigg]{ \log \frac{L_0 + L_1^{3/2} \Delta}{\sigmainit } }_+
    },
  \end{align}
  which completes the proof.
\end{proof}

\section{Conclusion}
We analyzed a parameter-free second-order method for finding an $(\epsilon, \delta)$-stationary point under a pointwise generalized smoothness assumption weaker than the two-point assumptions used in existing analyses.
The method is a minor variant of CRN, with a modified backtracking procedure for choosing the regularization parameter $\sigma$.
Specifically, the backtracking procedure checks the two acceptance conditions~\cref{eq:accept_decrease_condition,eq:accept_grad_condition} and starts each iteration with half of the previously accepted value of $\sigma$.
The method also allows inexact solutions of the cubic-regularized subproblem.
For this method, we established the oracle complexity bound in \cref{cor:epsilon_delta_oracle_complexity}, improving on existing guarantees for parameter-free second-order methods.
In the classical smoothness case $L_1 = 0$, this bound is optimal up to additive logarithmic terms.
This resolves the open question posed by \citep{hamad2022consistently} by showing that a parameter-free second-order method can achieve a sharp bound for finding an $(\epsilon, \delta)$-stationary point.

Our result is sharp when $L_1 = 0$, but whether it is sharp when $L_1 > 0$ remains open.
One important direction for future work is to determine whether the $\O(\Delta \sqrt{L_1} \epsilon^{-1})$ and $\O(\Delta L_1 \delta^{-1})$ terms in \cref{cor:epsilon_delta_oracle_complexity} are unavoidable when $L_1 > 0$.
Another direction, motivated by existing analyses of higher-order methods under generalized smoothness~\citep{jerad2026fast}, is to extend our analysis to such methods under weak pointwise smoothness assumptions similar to the one considered in this paper.

\section*{Acknowledgments}
This work was partially supported by JSPS KAKENHI (23K28041 and 24K23853) and JST CREST (JPMJCR24Q2).

\bibliographystyle{abbrvnat_nodoi}
\bibliography{reference.bib}

\end{document}

%% file: packages.tex
\usepackage{setspace}   % line spacing
\usepackage{pdflscape}  % landscape pages
\usepackage{microtype}  % microtypography
\usepackage{xcolor}     % colors

\usepackage{amsmath,amssymb,amsthm} % AMS math
\usepackage{mathtools}              % amsmath extensions
\usepackage{mathrsfs}               % \mathscr
\usepackage{latexsym}               % extra symbols
\usepackage{bm}                     % bold math
\usepackage{graphicx}                         % graphics inclusion
\usepackage{wrapfig}                          % wrapped figures
\usepackage[subrefformat=parens]{subcaption} % subcaptions

\usepackage{array,booktabs,multirow,diagbox,threeparttable,siunitx,tabularx}
\newcolumntype{P}[1]{>{\centering\arraybackslash}p{#1}} % centered p-column
\newcolumntype{C}{>{\centering\arraybackslash}X}        % centered X-column
\usepackage{pbox} % paragraph box

\usepackage{algorithm}
\usepackage[noend]{algpseudocode}

\algnewcommand{\Break}{\textbf{break}}
\algnewcommand{\algorithmicgoto}{\textbf{go to} Line}
\algnewcommand{\Goto}[1]{\algorithmicgoto~\ref{#1}}
\algblockdefx{MRepeat}{EndRepeat}{\textbf{repeat}}{}
\algnotext{EndRepeat}

\usepackage[sort&compress,numbers,square]{natbib}
\usepackage{thmtools,thm-restate}

\usepackage{url}       % \url
\usepackage{multicol}  % multi-column text
\usepackage{wasysym}   % extra symbols
\usepackage{adjustbox} % box adjustment
\usepackage{xparse}    % command definitions
\usepackage{authblk}   % author blocks
\usepackage{xspace}    % smart spaces
\usepackage{enumitem}  % list customization

\usepackage{hyperref}
\usepackage{cleveref}

\expandafter\def\csname ver@etex.sty\endcsname{3000/12/31}

\usepackage{autonum} % number only referenced equations

\makeatletter
\autonum@generatePatchedReferenceCSL{Cref} % fix cleveref + autonum
\makeatother

\hypersetup{
  colorlinks,
  linkcolor={green!35!black},
  citecolor={green!35!black},
  urlcolor={blue!50!black}
}

%% file: settings.tex
\declaretheoremstyle[
  shaded={bgcolor=gray!15},
]{thmsty}

\declaretheorem[
  name=Theorem,
  refname={Theorem,Theorems},
  style=thmsty,
]{theorem}

\declaretheorem[
  name=Corollary,
  refname={Corollary,Corollaries},
  style=thmsty,
]{corollary}

\declaretheorem[
  name=Lemma,
  refname={Lemma,Lemmas},
  style=thmsty,
]{lemma}

\declaretheorem[
  name=Assumption,
  refname={Assumption,Assumptions},
  style=thmsty,
]{assumption}

\declaretheorem[
  name=Remark,
  refname={Remark,Remarks},
  style=remark,
]{remark}

\declaretheorem[
  name=Example,
  refname={Example,Examples},
  style=definition,
]{example}

\crefname{algorithm}{Algorithm}{Algorithms}
\crefname{line}{Line}{Lines}
\crefname{section}{Section}{Sections}
\crefname{appendix}{Appendix}{Appendices}
\crefname{table}{Table}{Tables}
\crefname{figure}{Figure}{Figures}

\crefname{equation}{}{}
\Crefname{equation}{Eq.}{Eqs.}

\setlist[itemize]{
  topsep=0.1\baselineskip,
  itemsep=0\baselineskip,
  leftmargin=1.5em,
}

\setlist[enumerate]{
  font=\upshape,
  label=(\alph*),
  ref=(\alph*),
  topsep=0.1\baselineskip,
  itemsep=0\baselineskip,
  leftmargin=2em,
}

\newlist{enuminasm}{enumerate}{1}
\setlist[enuminasm]{
  font=\upshape,
  label=(\alph*),
  ref=\theassumption(\alph*),
  topsep=0.4\baselineskip,
  itemsep=0\baselineskip,
  leftmargin=2em,
}
\crefalias{enuminasmi}{assumption}

\newlist{enuminthm}{enumerate}{1}
\setlist[enuminthm]{
  font=\upshape,
  label=(\alph*),
  ref=\thetheorem(\alph*),
  topsep=0.4\baselineskip,
  itemsep=0\baselineskip,
  leftmargin=2em,
}
\crefalias{enuminthmi}{theorem}

\newlist{enuminlem}{enumerate}{1}
\setlist[enuminlem]{
  font=\upshape,
  label=(\alph*),
  ref=\thelemma(\alph*),
  topsep=0.4\baselineskip,
  itemsep=0\baselineskip,
  leftmargin=2em,
}
\crefalias{enuminlemi}{lemma}

\newlist{enuminprop}{enumerate}{1}
\setlist[enuminprop]{
  font=\upshape,
  label=(\alph*),
  ref=\theproposition(\alph*),
  topsep=0.4\baselineskip,
  itemsep=0\baselineskip,
  leftmargin=2em,
}
\crefalias{enuminpropi}{proposition}

\newlist{enumincond}{enumerate}{1}
\setlist[enumincond]{
  font=\upshape,
  label=(\alph*),
  ref=\thecondition(\alph*),
  topsep=0.4\baselineskip,
  itemsep=0\baselineskip,
  leftmargin=2em,
}
\crefalias{enumincondi}{condition}

%% file: macros.tex
\DeclareMathOperator{\EE}{\mathbb{E}}

\DeclarePairedDelimiter\ceil{\lceil}{\rceil}

\DeclarePairedDelimiterX\inner[2]{\langle}{\rangle}{{#1},{#2}}
\DeclarePairedDelimiter\abs{\lvert}{\rvert}
\DeclarePairedDelimiter\norm{\lVert}{\rVert}

\providecommand\given{}
\newcommand\SetSymbol[1][]{%
  \nonscript\:#1\vert
  \allowbreak
  \nonscript\:
  \mathopen{}%
}

\DeclarePairedDelimiter\prn{(}{)}
\DeclarePairedDelimiter\brk{[}{]}
\DeclarePairedDelimiter\set{\lbrace}{\rbrace}

\DeclarePairedDelimiterX\Prn[2]{(}{)}{%
  \renewcommand\given{\SetSymbol[\delimsize]}%
  #1\given#2%
}

\DeclarePairedDelimiterX\Brk[2]{[}{]}{%
  \renewcommand\given{\SetSymbol[\delimsize]}%
  #1\given#2%
}

\DeclarePairedDelimiterX\Set[2]{\lbrace}{\rbrace}{%
  \renewcommand\given{\SetSymbol[\delimsize]}%
  #1\given#2%
}

\newcommand{\N}{\mathbb N}

\newcommand{\R}{\mathbb R}

\newcommand{\0}{\mathbf 0}

\newcommand{\e}{\mathrm{e}}
\newcommand{\dd}{\mathrm{d}}
\renewcommand{\O}{\mathrm{O}}

\renewcommand{\epsilon}{\varepsilon}
\NewDocumentCommand{\ex}{s o d<> m}{%
  \IfBooleanTF{#1}{%
    \IfValueTF{#3}{\EE_{#3}\brk*{#4}}{\EE\brk*{#4}}%
  }{%
    \IfValueTF{#3}{%
      \IfNoValueTF{#2}{\EE_{#3}\brk{#4}}{\EE_{#3}\brk[#2]{#4}}%
    }{%
      \IfNoValueTF{#2}{\EE\brk{#4}}{\EE\brk[#2]{#4}}%
    }%
  }%
}

\NewDocumentCommand{\cex}{s o d<> m m}{%
  \IfBooleanTF{#1}{%
    \IfValueTF{#3}{\EE_{#3}\Brk*{#4}{#5}}{\EE\Brk*{#4}{#5}}%
  }{%
    \IfValueTF{#3}{%
      \IfNoValueTF{#2}{\EE_{#3}\Brk{#4}{#5}}{\EE_{#3}\Brk[#2]{#4}{#5}}%
    }{%
      \IfNoValueTF{#2}{\EE\Brk{#4}{#5}}{\EE\Brk[#2]{#4}{#5}}%
    }%
  }%
}

\usepackage{CJKutf8}

\newcommand{\mathInd}{\hphantom{{}={}}}
\newcommand{\by}[2][]{\text{\pbox[c]{\textwidth}{(by \pbox[t]{\textwidth}{\phantom{}#2)#1}}}}
\newcommand{\email}[1]{\href{mailto:#1}{\nolinkurl{#1}}}

\makeatletter
\renewcommand{\subparagraph}{%
  \@startsection{subparagraph}{5}{\parindent}%
    {1pt}% 見出し前の縦方向の空白
    {-1em}% 負の値なので run-in heading（見出し後に改行しない）
    {\normalfont\normalsize\bfseries}%
}
\makeatother

\usepackage[table]{xcolor}